\documentclass{amsart}

\usepackage{amsmath,amssymb,amsfonts,amstext,amsthm}
\usepackage{url}
\usepackage{graphicx}
\input xy
\xyoption{all}
\usepackage[all]{xy}

 \usepackage[abbrev]{amsrefs}

\graphicspath{{images/}{../images/}}
\usepackage{subfiles}
\usepackage{tikz}

\newcommand{\nwc}{\newcommand}

\nwc{\aaa}{\mathcal{F}}
\nwc{\aap}{\mathcal{F}_{P}}
\nwc{\al}{\alpha}

\nwc{\C}{\mathbb{C}}
\nwc{\cb}{\overline{C}}
\nwc{\ccc}{\mathcal{C}}
\nwc{\ch}{\widehat{C}}
\nwc{\cin}{\textbf{(v)}}
\nwc{\cl}{C'}
\nwc{\cp}{\mathcal{C}_{P}}
\nwc{\cpll}{\mathfrak{c}_{P'}}
\nwc{\ct}{\widetilde{C}}

\nwc{\dd}{\mathcal{L}}
\nwc{\ddd}{\mathfrak{d}}
\nwc{\ddl}{\mathcal{L}'}
\nwc{\dlp}{\delta_{P}}
\nwc{\doi}{\textbf{(ii)}}

\nwc{\enq}{$$}

\nwc{\fl}{\flushleft}
\nwc{\fff}{\mathcal{F}}
\nwc{\ffp}{\mathcal{F}_{P}}
\nwc{\ffq}{\mathcal{F}_{Q}}
\nwc{\ffl}{\mathcal{F}'}

\nwc{\G}{\mathcal{G}}
\nwc{\Ga}{\Gamma}
\nwc{\gon}{{\rm gon}}
\nwc{\gtl}{\widetilde{g}}

\nwc{\hra}{\hookrightarrow}
\nwc{\hua}{h^{1}(C,\aaa )}

\nwc{\kk}{{\rm K}}
\nwc{\kp}{{\kappa}}

\nwc{\llb}{\mathcal{L}}

\nwc{\mb}{\mathbb}
\nwc{\mc}{\mathcal}
\nwc{\mm}{\mathfrak{m}}
\nwc{\mmp}{\mathfrak{m}_{P}}
\nwc{\mpd}{\mathfrak{m}_{P}^{2}}

\nwc{\nn}{\mathbb{N}}

\nwc{\ob}{\overline{\mathcal{O}}}
\nwc{\obr}{\mathcal{O}^*}
\nwc{\obp}{\overline{\mathcal{O}}_P}
\nwc{\och}{\mathcal{O}_{\hat{C}}}
\nwc{\oh}{\widehat{\mathcal{O}}}
\nwc{\ohp}{\widehat{\mathcal{O}}_{P}}
\nwc{\ol}{\mathcal{O}'}
\nwc{\oma}{\Omega (\mathfrak{a})}
\nwc{\omo}{\Omega (\mathcal{O})}
\nwc{\oo}{\mathcal{O}}
\nwc{\op}{\mathcal{O}_P}
\nwc{\opc}{\mathcal{O}_{P,C}}
\nwc{\oph}{\hat{\mathcal{O}}_{P}}
\nwc{\opl}{\mathcal{O}_{P}'}
\nwc{\oplc}{\mathcal{O}_{P,C}'}
\nwc{\opll}{\mathcal{O}_{P'}}
\nwc{\opt}{\tilde{\mathcal{O}}_{P}}
\nwc{\optt}{{\mathcal{O}}_{\tilde{P}}}
\nwc{\oq}{\mathcal{O}_{Q}}
\nwc{\oqt}{\tilde{\mathcal{O}}_{Q}}
\nwc{\ot}{\widetilde{\mathcal{O}}}
\nwc{\overop}{\bar{\oo}_{P}}

\nwc{\pb}{\overline{P}}
\nwc{\pbb}{P^*}
\nwc{\pbi}{\overline{P_{i}}}
\nwc{\pbr}{\overline{P_{r}}}
\nwc{\pgmd}{\mathbb{P}^{g+2}}
\nwc{\pgmu}{\mathbb{P}^{g+1}}
\nwc{\ph}{\hat{P}}
\nwc{\pp}{\mathbb{P}}
\nwc{\ppn}{\mathbb{P}^{n}}
\nwc{\prv}{\noindent\textbook{Proof}:}
\nwc{\pt}{\widetilde{P}}
\nwc{\ptl}{\tilde{P}}
\nwc{\pum}{\mathbb{P}^{1}}

\nwc{\qh}{\hat{Q}}
\nwc{\qtl}{\tilde{Q}}
\nwc{\qua}{\textbf{(iv)}}

\nwc{\ra}{\rightarrow}
\nwc{\rh}{\hat{R}}

\nwc{\sei}{\textbf{(vi)}}
\nwc{\sep}{\beq\ast\ \ast\ \ast\enq}
\nwc{\sig}{\sigma}
\nwc{\Sig}{\Sigma}
\nwc{\ssp}{{\rm S}_{P}}
\nwc{\sss}{{\rm S}}
\nwc{\sssh}{\widehat{{\rm S}}}
\nwc{\sys}{\mathcal{L}}

\nwc{\tre}{\textbf{(iii)}}

\nwc{\um}{\textbf{(i)}}

\nwc{\val}{\mathcal{V}}
\nwc{\vpb}{v_{\overline{P}}}
\nwc{\vtxp}{\widetilde{V}_{x,P}}
\nwc{\vv}{\mathcal{W}}
\nwc{\vvp}{\mathcal{W}_{P}}
\nwc{\vxp}{V_{x,P}}

\nwc{\wh}{\hat{\omega}}
\nwc{\whp}{\hat{\omega}_{P}}
\nwc{\woch}{\omega\cdot\mathcal{O}_{\hat{C}}}
\nwc{\woh}{\omega\cdot\hat{\mathcal{O}}}
\nwc{\ww}{\omega}
\nwc{\wwb}{\omega^*}
\nwc{\wwct}{\omega _{\widetilde{C}}}
\nwc{\wwh}{\widehat{\omega}}
\nwc{\wwhp}{\widehat{\omega}_P}
\nwc{\wwp}{\omega _{P}}
\nwc{\wwt}{\widetilde{\omega}}
\nwc{\wwtp}{\widetilde{\omega}_P}

\nwc{\zz}{\mathbb{Z}}

\newtheorem{coro}{Corollary}[section]

\newtheorem{lemma}[coro]{Lemma}

\newtheorem{rem}[coro]{Remark}

\newtheorem{thm}[coro]{Theorem}

\newtheorem{ex}[coro]{Example}
\let \fl=\flushleft

\let \al=\alpha

\begin{document}

\title{On Gonality and Canonical Models of Unicuspidal Rational Curves}

\author{Naam\~a Galdino}
\address{Departamento de Matem\'atica, ICEx, UFMG
Av. Ant\^onio Carlos 6627,
30123-970 Belo Horizonte MG, Brazil}
\email{naamagaldino@ufmg.br}

\author{Renato Vidal Martins}
\address{Departamento de Matem\'atica, ICEx, UFMG
Av. Ant\^onio Carlos 6627,
30123-970 Belo Horizonte MG, Brazil}
\email{renato@mat.ufmg.br}

\author{Danielle Nicolau}
\address{Departamento de Matem\'atica, UFV / CAF,
Rodovia LMG 818 km 06,
35690-000 Florestal MG, Brazil}
\email{dani.nicolau@ufv.br}




\maketitle

\begin{abstract}
We study the gonality of a curve $C$ and its canonical model $C'$, by means of unicuspidal rational curves,  where those concepts can be better understood. We  start by a general formula for the dimension of the space of hypersurfaces of a fixed degree containing $C'$, which we apply to some particular cases. Then we  classify unicuspidal rational curves via different notions of gonality, and by its canonical model, up to genus 6. We do it using general methods applied to certain families of curves of arbitrary genus. 
\end{abstract}


\section*{Introduction}

Gonality is a fundamental measure of the (ir)rationality of a projective curve. The usual definition, in the smooth case, is the smallest $d$ for which the curve $C$: (a) admits a $d$-cover of $\mathbb{P }^1$, or, equivalently, (b) carries a pencil of degree $d$. Also, this notion has a clear geometric appeal: $C$ is $d$-gonal if and only if $C$ lies on a $(d-1)$-fold scroll whose rulings cut out the $g_d^1$ \cite{EH,Sc}. This result was extended in \cite{RS} for $d=3$ and $C$ singular Gorenstein, as in this case $C$ can still be viewed as a \emph{canonical curve}.


\medskip
The simplest non-Gorenstein curves are examples of how (a) and (b) may differ. They are defined as having a unique singularity for which the maximal ideal and the conductor coincide, and were studied, e.g., in \cite{BC, EHKS,St} and called \emph{nearly normal} in \cite{KM}. They carry a $g_2^1$ with a non-removable base point \cite[Thm.~2.1]{Mt} but the smallest degree of a cover to $\mathbb{P}^1$ is $g+1$, where $g$ is the arithmetic genus (a consequence of our Theorem \ref{thmggr}.(iii)). 

\medskip
Also, if $C$ is non-Gorenstein, then the expected gonality-scroll correspondence should be stated as follows: $C$ is $d$-gonal iff $C'$ lies on a $(d-1)$-fold scroll (cf. \cite{LMS}). Here, $C'$ is the \emph{canonical model} of $C$ introduced by Rosenlicht in \cite{R} and defined as follows: let $\cb$ be the normalization and $\ww$ the dualizing sheaf of $C$. Then the global sections of $\ww$ induce a morphism $\cb\to\pp^{g-1}$ whose image is $C'$. In other words, the canonical model $C'$ (and not $C$) is the curve where a $g_d^1$ on $C$ is geometrically realized as intersections with rulings of a scroll. 

\medskip
In \cite{FM}, a study of non-Gorenstein curves of $g\leq 5$ via gonality is carried out based on a systematic analysis of all possible semigroups. This is the point of departure of the present work. Motivated by \cite{CFM1, CLM}, if we restrict our study to unicuspidal rational curves,  it becomes feasible to identifying not only semigroups, but to classify curves via gonality, give parametric descriptions of canonical models, and, in particular, have a better and explicit understanding of these concepts.  .

\medskip
So we start deriving in Theorem \ref{thmdim} a general formula for the dimension of hypersurfaces of a fixed degree containing $C'$, when $C$ satisfies a certain parameter property, which we give a way of computing in Remark \ref{remrem}. The key ingredient is a singular version of the Noether Theorem recently proved in \cite{GM}. Then we verify the result for genus-$4$ unicuspidal rational curves. Since $C'$ is a space-curve if $g=4$, we are also able to depict it as an intersection of quadrics and cubic hypersurfaces. The quadrics correspond to $2$-fold scrolls containing $C'$, while the cubics determine the induced $g_3^1$ cut out by the rulings. In order to make computations, we rely on a parametric description of $C'$ done later on in the text.

\medskip
In Theorem \ref{thmggr}, we prove useful results for computing gonality and describing canonical models of rational curves with a unique cusp $P$ by means of semigroup techniques. We get: (i) a form for a sheaf that computes gonality; (ii) a formula for its degree; (iii) a lower bound for the gonality computed by base point free pencils, which agrees with the multiplicity of $P$; in particular, $g+1$ is the smallest degree of a cover $C\to\mathbb{P}^1$ if $C$ is nearly normal; (iv) a description of $C'$ in terms of the semigroup of $P$.

\medskip
Then we apply the methods above to a family of curves with a prescribed semigroup of just one block of values in Theorem \ref{general}, where we get a parametric description of $C'$ and the gonalities they may admit. We also state a similar result in Theorem \ref{general1}, with details in a supplementary link. The methods may work for any semigroup. We chose the ones here mainly because they provide a full classification of all unicuspidal rational non-Gorenstein curves of genus at most $6$ in terms of different notions of gonality, with description of canonical models, in Theorem \ref{thmgn6}.

\medskip
\noindent{\bf Acknowledgements.} We thank the Referee for many suggestions and corrections we buit into this version. This work is part of the first named author Ph.D thesis. The second named author is partially supported by FAPEMIG  RED-00133-21.We thank C. Carvalho, E. Cotterill, A. Fenandez and M. E. Hernandes for many suggestions as well.

\section{Linear Systems, Gonality and Canonical Models}
\label{cap1}

For this section (but not for the remainder), a \emph{curve} is an integral and complete one-dimensional scheme over an algebraically closed ground field. Let $C$ be a curve of (arithmetic) genus $g$ with structure sheaf $\oo_C$, or simply $\oo$, and $k(C)$ the field of rational functions. Let $\pi :\overline{C}\rightarrow C$ be the normalization map, set $\overline{\oo}:=\pi _{*}(\oo _{\overline{C}})$ and call $\ccc:=\mathcal{H}\text{om}(\overline{\oo},\oo)$, the conductor of $\overline{\oo}$ into $\oo$. Let also $\ww_C$, or simply $\ww$, denote the dualizing sheaf of $C$. A point $P\in C$ is said to be \emph{Gorenstein} if $\ww_P$ is a free $\oo_P$-module. The curve is said to be \emph{Gorenstein} if all of its points are so, or equivalently, $\ww$ is invertible. It is said to be \emph{hyperelliptic} if there is a morphism $C\rightarrow\pp^1$ of degree $2$.

\subsection{Linear Systems and Gonality}
\label{secs21}
 
A \emph{linear system of dimension $r$ and degree $d$ on $C$} is a pair of the form $\dd(\aaa ,V)$ 
where $\mathcal{F}$ is a torsion free sheaf of rank $1$ and degree $d$ on $C$ and $V$ is a vector subspace of $H^{0}(\aaa )$ of dimension $r+1$. For short, $\dd$ is also referred as a $g_{d}^{r}$. Varying $x\in V$, a linear system may be seen as a set of exact sequences $\{\mathcal{H}{\text om}(\aaa,\ww)\stackrel{x^*}\rightarrow\ww\to\mathcal{Q}_x\}$ as in \cite{KK} or a set of coherent fractional ideal sheaves $\{x^{-1}\aaa\}$ as in \cite{S}, with $\aaa$ embedded in the constant sheaf $\mathcal{K}$ of rational functions (see \cite[Sec. 1]{LMS} for more details).

\medskip
The linear system is said to be \emph{complete} if $V=H^0(\aaa)$, written $|\aaa|$ when so. The \emph{gonality} of $C$ is the
smallest $d$ for which there exists a $g_{d}^{1}$ on $C$, or equivalently, for which there exists an $\aaa$ on $C$ as above with $\deg(\aaa)=d$ and $h^0(\aaa)\geq 2$. A point $P\in C$ is called a \emph{base point of $\dd$} if $x_P:\op\to\aaa_P$ is not an isomorphism for every $x\in V$. A base point is said \emph{removable} if it is not a base point of $\dd(\oo\langle V\rangle,V)$. 
So $P$ is a non-removable base point of $\dd$ if and only if $\aaa_P$ is not a free $\op$-module. In particular, $P$ is singular if so.

\subsection{The Canonical Model}
\label{seccan}


Consider the morphism $\kappa :\overline{C}\rightarrow\pp^{g-1}$ induced by the linear system $\sys(\pi^*\ww /\text{Tors}(\pi^*\ww),H^0(\ww))$. The image $C':=\kappa(C)$ is the \emph{canonical model} of $C$. Rosenlicht proved in \cite[Thm\,17]{R} that if $C$ is nonhyperelliptic, then there is a morphism
$C'\rightarrow C$ factoring the normalization map. 

\medskip
Now set $\ww^n:=\text{Sym}^n\ww/\text{Tors}(\text{Sym}^n\ww)$ and consider the \emph{canonical blowup} $\widehat{C}:=\text{Proj}(\oplus\,\ww ^n)$ of $C$ along $\ww$. In \cite[Thm.~6.4]{KM}, Rosenlicht's result is reproved and rephrased as an isomorphism $\widehat{C}\stackrel{\thicksim}{\to}C'$.

\medskip
The canonical model clearly plays the role of a \emph{canonical curve}. If $C$ is non-hyperelliptic Gorenstein, then $C'$ is nothing but the embedding of $C$ on $\pp^{g-1}$ via the canonical map. Otherwise $C'$, though failing being isomorphic to $C$, carries much of relevant data of the original curve, as for instance the realization of linear series via scrolls (cf. \cite{LMS}).

\subsection{Numerical Semigroups} 

Let $P\in C$ be a unibranch point, i.e., there is a unique $\pb\in\cb$ lying over $P$. Let also
$$
v_{\pb}: k(C)^* \longrightarrow \zz
$$
be the valuation map. The \emph{numerical semigroup} of $P$ is
$$
\sss=\ssp:=\vpb(\op ).
$$
Now set
\begin{equation*}
\label{equaab}
\alpha=\alpha_P :={\rm min}(\sss\setminus\{ 0\})\ \ \ \ \ \text{and}\ \ \  \ \ \ \beta = \beta_P ={\rm min}\{a\in\sss\,|\, a+\mathbb{N}\subset \sss\}
\end{equation*}
These parameters are related to important dimensions, namely,
$$
\alpha=\mbox{dim}(\obp/\mmp\obp)\ \ \ \ \ \ \text{and}\ \ \ \ \ \ \beta=\mbox{dim}(\obp/\mathcal{C}_P)
$$
where
$$
\cp:=\{ x\in k(C)\,|\, x\obp \subset \op\}
$$
is the \emph{conductor} ideal of $P$. One defines the set of \emph{gaps} of $\sss$ as
$$
{\rm G} :=\nn\setminus\sss = \{\ell_1=1,\ell_2,\ldots,\ell_{\delta}=\beta-1\} 
$$
where
\begin{equation}
\label{equkkd}
\delta :=\#({\rm G}) 
\end{equation}
agrees with the singularity degree of $P$, that is, 
\begin{equation}
\label{equdlt}
\delta=\dim(\obp/\op).
\end{equation} 

One also defines
\begin{equation*}
\label{equkkp}
\kk=\kk_{P}:=\{ a\in\zz\ |\ \gamma -a\not\in\sss\}
\end{equation*}
To see its importance, recall one also defines a valuation map on differentials
$$
\vpb : \Omega_{k(C)|k} \longrightarrow \mathbb{Z}
$$
via (any choice of) a local parameter $t$ of $\pb$ by $\vpb(xdt):=\vpb(x)$. Then we have:

\begin{lemma}
\label{lemlbd}
Let $P\in C$ be unibranch. Then there exists $\lambda\in H^0(\ww)$ such that:
\begin{itemize}
\item[(i)] $\ww_P/\lambda = H^0(\ww)/\lambda + \cp$;
\item[(ii)] $\vpb(\lambda) = -\beta$.
\item[(iii)] $\vpb(\wwp/\lambda)=\kk$
\end{itemize}
\end{lemma}

\begin{proof}
Item (i) follows from the last paragraph of the proof of \cite[Lem. 6.1]{KM}; item (ii) follows from \cite[Lem. 2.8]{KM}; and item (iii) follows from \cite[Thm. 2.11]{S}.
\end{proof}


\section{Computing the Canonical Model}
\label{sechyp}
 
In this section we are particularly interested on the study of the canonical model $C'$. We first derive a general result for arbitrary curves, not necessarily rational and unicuspidal. And then apply it to particular simple examples, where we are also able to depict the curves we are dealing with.

\subsection{Hypersurfaces containing the Canonical Model} To begin with, we fix some notation. Following \cite[Def.~2.7]{KM}, set
$$
\eta_P:=\delta_P-\dim(\op/\mathcal{C}_P)
$$
and $\eta :=\sum_{P\in C}\eta_P$. Let $\widehat{C}:=\text{Proj}(\oplus\,\ww ^n)$ be the canonical blowup (defined in Sec.~\ref{seccan}) and $\widehat{\pi}:\ch\to C$ the natural map. Thus set 
$\oo_{\ch}\ww:=\widehat{\pi}^*\ww/{\rm Torsion}(\widehat{\pi}^*\ww)$ and $\oh\ww:=\widehat{\pi}_*(\oo_{\ch}\ww)$. We define
$$
\sigma:={\rm min}\{n\in\nn\,|\, \ww^n=(\oh\ww)^n\}
$$
Set also
$$
I_n(C'):=\{\text{hypersurfaces of degree $n$ on $\mathbb{P}^{g-1}$ containing $C'$}\}
$$
With this in mind we prove the following result.

\begin{thm}
\label{thmdim}
Let $C$ be an integral and projective curve of genus $g$ with canonical model $C'$ of genus $g'$. Assume $n\geq\sigma$. Then
\begin{itemize}
\item[(i)] The natural map
$$
{\rm Sym}^n H^0(\ww) \longrightarrow H^0(\oo_{C'}(n))
$$
is surjective
\item[(ii)] 
$$
\dim(I_n(C'))=\binom{g+n-1}{n}-n(2g-2-\eta)-1+g'
$$
\end{itemize}
\end{thm}

\begin{proof}
The linear series $\dd(\oo_{\ch}\ww,H^0(\ww))$ yields an isomorphism $\widehat{\kappa}: \widehat{C}\to C'$ owing to \cite[Thm. 6.4]{KM} with $\oo_{\ch}\ww=\widehat{\kappa}^*\oo_{C'}(1)$. This allows us to address the problem intrinsically, that is, $H^0(\oo_{C'}(n))=H^0((\oh\ww)^n)$. Now for $n\geq \sigma$, we have that $\ww^n=(\oh\ww)^n$ and, in particular, $H^0((\oh\ww)^n)=H^0(\ww^n)$. Now, by the singular version of the Noether Theorem, proved in \cite{GM}, we have that ${\rm Sym}^n H^0(\ww) \rightarrow H^0(\ww^n)$ is surjective for any $n\geq 1$, so (i) follows.

\medskip
To get (ii), we just have to compute $h^0((\oh\ww)^n)$. From \cite[Cor. 4.9]{KM}, we have that $\deg(\oh\ww)=2g-2-\eta$. On the other hand, $h^1((\oh\ww)^n)=0$ for every $n\geq 1$ owing to \cite[Prp.~5.2]{KM}. Thus 
\begin{align*}
\dim(I_n(C'))&=\dim({\rm Sym}^n H^0(\ww))-\dim(H^0(\oh\ww)^n)\\
             &=\binom{g+n-1}{n}-n(2g-2-\eta)-1+g'
\end{align*}
where the second equality owes to Riemmann-Roch and the fact that $\oh\ww$ is a line bundle (even if $\ww$ is not) and hence $\deg(\oh\ww)^n=n\deg(\oh\ww)$.
\end{proof}

\begin{rem}
\label{remrem}
\emph{There is a very useful way to compute the parameters that appear in the statement of the theorem. First off, notice that $\ww^n$ and $(\oh\ww)^n$ agree in the Gorenstein locus of $C$ for any $n\geq 1$. So we are reduced to check non-Gorenstein points. Lemma \ref{lemlbd} provides an embedding $\ww\hookrightarrow \mathcal{K}$ via $\lambda\in H^0(\ww)$ such that for any singular point $P\in C$, we have
$$
\cp\subset\op\subset \wwp \subset \oh_P\subset \obp
$$ 
Now, plainly, $(\oh\ww)_P=\oh_P$. Thus $((\oh\ww)^n)_P=\oh_P^n=\oh_P$. On the other hand, by the very definition of the canonical blowup, $\oh_P$ is the smallest subring of $\obp$ containing $\ww_P$. But, as $\vpb(\wwp)=\kk$, it follows that $\sigma_P = {\rm min}\{n\in\nn\, |\,  n\kk = \langle \kk \rangle\}$, where $n\kk$ stands for the set of sums of $n$ elements of $\kk$, while $\langle \kk \rangle$ is the semigroup generated by $\kk$. And $\sigma$ is the largest among the $\sigma_P$. Also, \cite{KM} provides an alternative way of computing the parameter $\eta_P$, namely, $\eta_P=\dim(\wwp/\op)$, from which we get that $\eta_P=\#(\kk\setminus\sss)$.}
\end{rem}

\subsection{Examples} We now give a graphic description of the canonical model $C'$ in the first four cases of Theorem \ref{thmgn6} ahead. The techniques of describing $C'$ parametrically are developed in next section. What the examples have in common is that we are able to draw both the curve and (at least two) surfaces containing it, as $C'$ is either in the plane ($g=3$) or in three-dimensional space ($g=4$). We recall by \cite{LMS}, that if $C$ is $d$-gonal, then $C'$ is contained in a $(d-1)$-fold scroll in $\mathbb{P}^{g-1}$ of which the rulings cut out a $g_d^1$ on $C'$. So in the first case, $C'\subset\mathbb{P}^2$, which is the scroll itself, and the rulings are a pencil of lines meeting the curve generically at $3$ points. For the other three cases, $C'$ is a curve in $\mathbb{P}^3$ contained in a $2$-fold scroll (which is a quadric) and the rulings meet $C'$ generically in $3$ points as well since all these cases correspond to trigonal curves. 

\medskip
\noindent\textbf{case (i):} $C=(1+at\, :\, t^3 :\, t^5\, :\, t^6:\, t^7)$ and  $C'=(1+at\, :\, t^2\,:\,t^3)\subset\mathbb{P}^2$. Writing $\mathbb{P}^2=\{(w:x:y)\}$, we have that $C'$ is given by the cubic $x^3+ax^2y-y^2w$. 

\medskip
\noindent\textbf{case (ii):} $C=(1+at+bt^2\, :\, t^3 :\, t^7\, :\, t^8)$ (of genus 4) and the canonical model is $C'=(1+at+bt^2\, :\, t+at^2\, :\, t^3\,:\,t^4)\subset\mathbb{P}^3$ $\{(w:x:y:z)\}$. Clearly, $C'\cong\pum$, so $g'=0$. Also, as $C$ has just one singular point $P$, we have that $\eta=\eta_P$ and it can be computed as $\eta=\dim(\wwp/\op)=\#(\kk\setminus\sss)=\#\{1,4\}=2$. Set $\sssh:=\vpb(\ohp)$. We have that $\sssh$ agrees with the semigroup generated by $\kk$. Besides, $\sssh\setminus\kk=\{2,5\}$ and both numbers are expressed as sum of two elements of $\kk$, namely,  $2=1+1$ and $5=1+4$. It follows that $\wwp^2=\ohp$ and $\sigma=2$. So we are in position to apply Theorem \ref{thmdim}. We have that $\dim(I_2(C'))=\binom{4+2-1}{2}-2(2\times 4-2-2)-1=1$ and 
$$
I_2(C')=k(by^2+xy-zw) 
$$
Now $\dim(I_3(C'))=\binom{6}{3}-3\times 4-1=7$ and one computes
\begin{align*}
I_3(C')&= k(x^3-axyw+2bxzw+b^3y^3-aby^2w-b^2yzw-yw^2) \\
&\ \ \ \oplus k(x^2y-xzw-b^2y^3+byzw)\oplus k(x^2z-b^2y^2z-y^2w-ayzw+bz^2w)\\
&\ \ \ \oplus k(xy^2+by^3-yzw)\oplus k(xyz+by^2z-z^2w)\oplus k(xz^2-y^3-ay^2z)\\
&\ \ \ \oplus k(-xyw-by^2w+zw^2) 
\end{align*}

\medskip
\noindent\textbf{case (iii):} $C=(1+at+bt^2\, :\, t^4 :\, t^5\, :\, t^7\, :\, t^8)$, of genus 4 as well, and the canonical model is $C'=(1+at+bt^2\, :\, t^3\, :\, t^4\,:\,t^5)\subset\mathbb{P}^3$ $\{(w:x:y:z)\}$. Here $C'$ is the unique (up to isomorphism) unicuspidal non-Gorenstein curve of genus $2$. Also, $\kk\setminus\sss=\{3\}$ so $\eta=1$ and $\sssh\setminus\kk=\{6\}$ so $\sigma=2$. Therefore $\dim(I_2(C'))=10-2(2\times 4-2-1)-1+2=1$ and 
$$
I_2(C')=k(y^2-xz) 
$$
while $\dim(I_3(C'))=20-3\times 5-1+2=6$ and one computes
\begin{align*}
I_3(C')& = k(x^2z-xy^2)\oplus k(xzw-y^2w)\oplus k(-xyz+y^3)\oplus k(y^2z-xz^2)\\
&\ \ \ \oplus k(yzw-x^3-ax^2y-bxy^2)\oplus k(z^2w-x^2y-axy^2-bxyz) 
\end{align*}

\medskip
\noindent\textbf{case (iv):} $C=(1+at\, :\, t^4 :\, t^6\, :\, t^7\, :\, t^9\, :\, t^{10})$, also of genus 4, and the canonical model is $C'=(1+at\, :\, t^2\, :\, t^3\,:\,t^4)\subset\mathbb{P}^3$ $\{(w:x:y:z)\}$. Now $C'$ is the ordinary cusp of genus $1$. Besides, $\kk\setminus\sss=\{2,3\}$ so $\eta=2$ and $\sssh\setminus\kk=\{5\}$ so $\sigma=2$. Thus $\dim(I_2(C'))=10-2(2\times 4-2-2)-1+1=2$ and 
$$
I_2(C')=k(y^2-xz)\oplus k(zw-x^2-axy)
$$
while $\dim(I_3(C'))=20-3\times 4-1+1=8$ and we have
\begin{align*}
I_3(C')&:= k(xy^2-x^2z)\oplus (xzw-x^3-ax^2y)\oplus k(y^3-xyz)\oplus k(y^2z-xz^2)\\
&\ \ \ \oplus k(y^2w-x^3-ax^2y)\oplus k(yzw-x^2y-ax^2z)\oplus k(z^2w-x^2z-axyz)\\
&\ \ \ \oplus k(zw^2-x^2w-axyw)
\end{align*}

As examples, we picture $C'$ contained in (ii) $\{ z=xy\}\cap\{ y^3=xz^2-y^2z\}$,\\ (iii) $\{ y^2=xz\}\cap\{ yz=x^3-x^2y\}$, and (iv) $\{ z=x^2+xy\}\cap\{ y^2=x^3+x^2y\}$ 

 \begin{figure}[h]
 \centering
\includegraphics[scale=0.6]{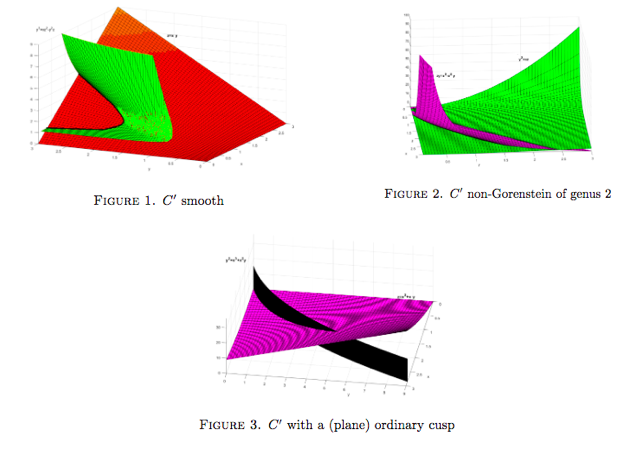}
\end{figure}

\section{Gonality and Parametric Description of Canonical Models}
\label{cap2}

In this section we develop a method of computing the gonality of a rational unicuspidal curve and describe parametrically its canonical model in Theorem \ref{thmggr}. Then we apply it to a family of curves with a prescribed singularity with one block of values in Theorem \ref{general}, also stating a sibling result in Theorem \ref{general1}. The method can be applied to any rational curve with a singleton singularity. The ones we choose target a full classification of unicuspidal rational curves of genus at most $6$ in Theorem \ref{thmgn6} in next section.



\medskip
To begin with, and for the sake of simplicity, if a parametrized curve is given by
$$
C=(F_0(t,s):\ldots :F_{N}(t,s))\subset \mathbb{P}^N
$$
where the $F_i$ are homogeneous polynomials, we write it just as 
\begin{equation}
\label{equpar}
C=(f_0(t):\ldots :f_N(t))\subset \mathbb{P}^N
\end{equation}
where $f_i(t)=F_{i}(t,1)\in k[t]$.

\medskip 
If $C$ as above is unicuspidal, we may also give it a totally intrinsic description. Indeed, so let $P$ be the unique singular point of $C$. Let $\pb\in \cb=\pum$ denote the preimage of $P$ under the normalization map $\pi:\pum\to C$. 
Write $k(C)=k(t)$ and $\pum=\C\cup\{\infty\}$ so that $t$ is the identity function at finite distance, and $\pb=0$. 

\medskip
Note that this intrinsic approach mathces the extrinsic one (\ref{equpar}) in the sense that the parameter $t$ is the same, that is, the ``abstract" local ring of $P$ reads
\begin{equation}
\label{equsum}
\oo_P=k+ k\, \frac{f_{1}(t)}{f_0(t)}+ \cdots + k\, \frac{f_{N}(t)}{f_0(t)}+ \cp 
\end{equation}
We also make the convention that $Q$ will always denote the point of $C$ at ``infinity", that is, $Q:=\pi(\infty)$ according to our convention. We just warn the reader not to make confusion with the points of $C$ which lie at the ``hyperplane at infinity" on $\mathbb{P}^N$, that is, the images $\varphi(c_i:1)$ of the roots $c_i$ of $F_0$, where the latter morphism is given by $\varphi:(t:s)\mapsto (F_0(t,s):\ldots :F_{N}(t,s))$. 

\medskip
We also adopt the following notation: given a space $V=\langle f_1,\ldots,f_s\rangle \subset k(C)$, the sheaf  $\aaa=\oo_C\langle V\rangle=\mathcal{O}_C \langle f_1,\ldots,f_s\rangle$ stands for the subsheaf of the constant sheaf of rational functions $\mathcal{K}$ generated by the sections $f_i$.

\medskip
Finally, we set $\gon(C)$ for the gonality of $C$; $\gon_F(C)$, for the smallest $d$ such that there exists a degree-$d$ cover $C\to\pum$, and $\gon_B(C)$ for the smallest $d$ such that $C$ carries a $g_d^1$ with a non-removable base point. Then $\gon(C)={\rm min}(\gon_B(C),\gon_F(C))$ 




\begin{thm}
\label{thmggr}
Let $C$ be a unicuspidal rational curve whose singularity $P$ has multiplicity $\alpha$ and conductor $\beta$. Let $\aaa$ be a torsion sheaf of rank $1$ on $C$. Then:
\begin{itemize}
\item[(i)] If $\aaa$ computes $\gon(C)$ or $\gon_F(C)$ then $\aaa=\mathcal{O}_C\langle1,t^r f/h\rangle$, with $f,h\in k[t]$ with ${\rm gcd}(f,h)=1$ and $f(0)\neq0\neq h(0)$, and $r>0$.
\item[(ii)] We have that
\begin{equation}\label{desi1}
\deg_{C\setminus P}(\aaa)=\deg(h)+\text{max}\{0,r+\deg(f)-\deg(h)\}.
\end{equation}
\item[(iii)] $\gon_F(C)\geq \alpha$; in particular, $\gon_F(C)=g+1$ if $C$ is nearly normal.
\item[(iv)] The canonical model of $C$ is of the form:
$$
C'=(f_0:\ldots:f_{g-1})\subset \mathbb{P}^{g-1}
$$
with $\{v_{\pb}(f_i)\}=\kk^{\circ}:=\{a\in\kk\,|\,a<\beta\}$ and $f_{g-1}=t^{\beta-2}.
$
\end{itemize}
\end{thm}

\begin{proof}
To compute the gonality, as seen in Section \ref{secs21}, we look for the smallest $d$ for which there exists a torsion free sheaf $\aaa$ on $C$ of rank $1$, such that $\deg(\aaa)=d$ and $h^0(\aaa)\geq 2$. 
First, if $\aaa$ computes the gonality of $C$, then it is generated by global sections, since otherwise $\oo_C\langle H^0(\aaa)\rangle$ is a sheaf of smaller degree satisfying $h^0(\aaa)\geq 2$.  Also, the pencil induced by $\aaa$ is complete. Indeed, for later use in the proof we follow \cite{LMS}: choose any regular point $R\in C$ and consider the sequence
$$
0\longrightarrow \aaa(-R)\longrightarrow \aaa \longrightarrow \aaa/\aaa(-R)\longrightarrow 0
$$
Taking Euler characteristic yields 
$$
\big(h^0(\aaa)-h^0(\aaa(-R))\big)+\big(h^1(\aaa(-R))-h^1(\aaa)\big)=1
$$
Now both summands are nonnegative; besides, $h^0(\aaa)\geq 2$ and we also have that $h^0(\aaa(-R))\leq 1$ because this sheaf is of degree $d-1$ and the gonality of $C$ is $d$. Thus $h^0(\aaa)=2$ as desired.  

\medskip
Therefore, if $\aaa$ computes the gonality of $C$, then it has two independent global sections which must generate it. But any torsion free sheaf of rank $1$ with a global section can be embedded in the constant sheaf of rational functions $\mathcal{K}$ so that
\begin{equation}
\label{equtfs}
\oo \subset \aaa \subset \mathcal{K}
\end{equation}
So we may write  $\aaa=\mathcal{O}_C\langle1,z\rangle$ where $z=t^r f/h$, and $f,h\in k[t]$ have non common factors no one having zero as a root either. We may first assume that $r\geq 0$, i.e, $z$ does not have a pole on $P$. Indeed, as our concern is gonality, we are mainly interested on $\deg(\aaa)$. But note that $\deg(x\aaa)=\deg(\aaa)$ for any $x\in k(C)$, thus we may pass to $z^{-1}\aaa=\oo_C\langle 1,z^{-1}\rangle$ if $r<0$. We may further assume $r>0$ just replacing $z$ by $z-z(0)$ if necessary, so (i) is proved.

\medskip
To prove (ii), the degree is defined as $\deg(\aaa)=\chi(\aaa)-\chi(\oo)$, but applying (\ref{equtfs}). we are reduced to the formula
\begin{equation}
\label{equgra}
\deg(\aaa)=\sum_{R\in C} {\rm length}(\aaa_R/\oo_R)
\end{equation}
owing to \cite[pp.~107-108]{S}.

\medskip
Hence, to compute the degree of $\aaa=\mathcal{O}_C\langle1,t^r f/h\rangle$ via (\ref{equgra}), the local length at each $R\in C$ may be computed as follows
$$
\deg_R(\aaa):={\rm length}({\aaa}_R/\oo_R)=\dim_k ((\oo_R+t^r f/h\cdot\oo_R)/\oo_R).
$$
But from our conventions, the regular points of $C$, i.e., the points in $C\setminus\{P\}$ are $Q$ and the $\pi$-image of each nonzero $c\in\C$. Writing $h=(t-c_1)^{m_1}\ldots(t-c_l)^{m_l}$ yields 
$$
\deg_R(\aaa)=
\begin{cases}
m_i & \text{if}\ R=\pi(c_i) \\
0 & \text{if}\ R\in C\setminus\{P,\pi(c_1),\ldots,\pi(c_l)\} \\
{\rm max}\{0,r+\deg(f)-\deg(h)\} & \text{if}\ R=Q
\end{cases}
$$
Summing up we get the formula stated in theorem.

\medskip
To prove (iii), suppose now $d$ is the smallest integer for which, there exists a degree-$d$ cover $C\to \pum$ or equivalently, for which there exists an invertible sheaf $\aaa$ such that $\deg(\aaa)=d$ and $h^0(\aaa)\geq 2$. First, $\aaa$ must be generated by global sections because $\oo_C\langle H^0(\aaa)\rangle$ is also a bundle and the same argument above applies. Also, if $\aaa$ is invertible, so is $\aaa(-R)$ because $R$ is regular, and we can use the same argument above to see that $h^0(\aaa)=2$. And finally, we may further assume $\aaa_P=\oo_P$. Indeed,  as $\aaa$ is locally free, then $\aaa_P=w\op$ for some $w\in k(C)$; so one replaces $\aaa$ by $w^{-1}\aaa$ if necessary. Therefore we can write $\aaa=\oo_C\langle 1,z\rangle$, with $z=t^rf/h$ as in (i), and such that $z\in\oo_P$ and $r\geq \alpha$.

\medskip
Now set $D:=\vpb(\aaa_P)\setminus \sss$. We have by \cite[Prp.~2.11]{BDF}
\begin{equation}
\label{equddd}
\deg_P(\aaa)=\#D 
\end{equation}
But as $z\in\op$, we get $\deg_P(\aaa)=0$. So (\ref{equddd}) yields
\begin{equation}
\label{equdbf}
d=\deg(h)+{\rm max}\{0,r+\deg(f)-\deg(h)\}
\end{equation}
If $r+\deg(f)-\deg(h)\leq 0$, then $d=\deg(h)\geq r+\deg(f)\geq \alpha$. Otherwise $d=r+\deg(f)\geq \alpha$ as well and we are done. To finish the proof, just note that if $C$ is nearly normal then $\alpha=\beta=g+1$.

\medskip
To get (iv), as priorly seen, $C'$ is the image of the morphism $\kappa :\overline{C}\rightarrow\pp^{g-1}$ induced by the linear series $\sys(\oo_{\overline{C}}\ww,H^0(\ww))$. But in our case $\cb=\pum$. So embedding $\ww$ in $\mathcal{K}$ and selecting $g$ independent global sections, say $\{ 1,F_1,\ldots F_{g-1}\}_{i=0}^{g-1}\subset k(t)$, the canonical model may be written as $C'=(f_0:\cdots :f_{g-1})$ where $F_i=f_i/f_0$.

\medskip
Now we will see how $\ww$ behaves locally. 
Embed $\ww$ on $\mathcal{K}$ via $\lambda$ as in Lemma \ref{lemlbd}, that is, if no confusion is made, replace $\ww$ by $\ww/\lambda$. 
Lemma \ref{lemlbd}.(iii) yields
\begin{equation}
\label{equeq1}
\vpb(\ww_P)=\kk
\end{equation}
On the other hand, by Lemma \ref{lemlbd}.(i) we have
\begin{equation}
\label{equeq2}
\ww_P = H^0(\ww) + \cp
\end{equation}
Now we claim that $\#(\kk^{\circ})=g$. Indeed, $\#\kk^{\circ}=\#(G)$  by the very definition of $\kk$. But $\#(G)=\delta_P$, the singularity degree of $P$ by (\ref{equdlt}). Since $\overline{g}=0$ and $P$ is a singleton singularity, the genus formula $g=\overline{g}+\sum_{R\in C}\delta_R$ \cite[Eq.~(1.3)]{S} yields $\delta_P=g$ and the claim follows. So combining (\ref{equeq1}) and (\ref{equeq2}) we get 
$$
\vpb(H^0(\ww))=\kk^{\circ}
$$
So we just may adjust the $f_i$ so that $\{v_{\pb}(f_i)\}=\kk^{\circ}$. To finish the proof we recall that 
by \cite[Prp. 2.14]{KM}, the canonical model is a curve of degree $\beta-2$ in $\mathbb{P}^{g-1}$. But as $\beta-2\in\kk^{\circ}$, we must have $f_{g-1}=t^{\beta-2}$. \end{proof}

\medskip
The theorem above is certainly helpful to compute gonality and canonical models of any unicuspidal rational curve
But those methods apply once the parametric curve is given. So another interesting and independent problem is finding a suitable parametrization for a curve of which we know just the semigroup $\sss$. For instance, in Theorem \ref{general} ahead, there is an additional effort to look for a helpful parametrization for our purposes. The reasons why are two. By Theorem \ref{thmggr}.(iii), $C'$ is of degree $\beta-2$ in $\mathbb{P}^{g-1}$, hence it is better to get rid of powers like $t^{\beta-1}$ in the parametrization of $C$ to describe its $C'$. Also, to compute gonality, it's more convenient to concentrate on the coefficients in the first component. For instance, the plane curves $(1+t^3:t^2:t^4)$ and $(1:t^2+t^3:t^4)$ are clearly isomorphic but it is somewhat easier to see that it's trigonal in the first case, by taking the  section $1/(1+t^3)$ which yields a line bundle of degree 3. On the other hand, wasting all coefficients in the first component may fail to get all unicuspidal curves with semigroup $\sss$. So what we do is balancing, which is the motivation behind the parametrization (\ref{equcv0}) in the statement of Theorem \ref{general}.

\begin{ex}
\label{exeexe}
\emph{We now give an example hoping it will be helpful in the proof of Theorem \ref{general}, of which we follow the same notation for coefficients indexation. Consider the semigroup $\sss=\{0,4,5,8,\to \}$ which, according to the theorem statement, has parameters $\ell=2$ (nunber of consecutive values) and $m=2$ (number of last consecutive gaps).
If $C$ is rational with a singleton cusp $P$ with semigroup $\sss$ then 
\begin{equation}
\label{equrng}
\op=k\oplus k f_1(t)\oplus k f_2(t)\oplus t^8 \obp 
\end{equation}
where $f_1(t),f_2(t) \in k(t)=k(C)$. We may further take $f_1=t^4+b_{11}t^6+b_{12}t^7$ and $f_2=t^5+b_{21}t^6+b_{22}t^7$, and $C$ is of the form
$$
C=(1:f_1:f_{2}:t^{10}:t^{11})
$$
where the $t^{11}$ is added to get $\sss$ while $t^{10}$ to avoid a singularity at infinity. Theorem \ref{general} provides a way of passing to a parametrization of the form
$$ 
C=(1+a_1t+a_2t^2+a_3t^3\, :\, t^4+a_4t^6\, :\, t^5\, :\,t^{10} :\,t^{11}) \subset \mathbb{P}^4
$$
Essentially one writes $1/(1+a_1t+a_2t^2+a_3t^3)$ as a power series and then gets the $b_{ij}$ from the $a_{k}$ and conversely without escaping from the local ring.}

\medskip
\emph{Let us find the gonality of $C.$ First, recall that $\gon(C)\geq 3$ since $C$ is not nearly normal \cite[Thm.~2.1]{Mt}. Assume $\gon(C)=3$ computed by a non-locally free sheaf of the form $\mathcal{F}=\mathcal{O} \langle 1, t^rf/h \rangle$ as in Theorem \ref{thmggr}.(i). Note that  $\deg_P(\aaa_P)\geq 1$ since $\aaa$ is not a bundle. Then $\deg_{C\setminus P}(\aaa)\leq 2$.
Thus $r+\deg(f)\leq 2$ owing to (\ref{desi1}), and hence $r\leq 2$. If $r=2$ then $\{2,7\}\in D$  and thus $\deg_P(\aaa)\geq 2$ owing to (\ref{equddd}).  Hence $\deg_{C\setminus P}(\aaa)\leq 1$ and all possibilities are precluded by (\ref{desi1}). If $r=1$ then $\deg_P(\aaa)\geq 2$, so $\deg_{C\setminus P}(\aaa)\leq 1$. From (\ref{desi1}) we get $\deg(f)=0$, $\deg(h)\leq 1$ and $\deg_{C\setminus P}\mathcal{F}=1$. Therefore $\deg_P(\aaa)=2$. Now $\{1,6\}\subset D$ if $r=1$ and equality holds if and only if $7\not\in D$. So we have to check this to assure that $\aaa$ has degree $3$.}

\medskip
\emph{Thus take $\mathcal{F}=\mathcal{O} \langle 1, t/(at+b) \rangle$ with $b\neq 0$. A general element in $\aaa_P$ is written
\begin{align*}
f(t)&=c_0+c_4\frac{t^4+a_4t^6}{1+a_1t+a_2t^2+a_3t^3}+c_5\frac{t^5}{1+a_1t+a_2t^2+a_3t^3}+\\
&d_0\frac{t}{at+b}+d_4\frac{t^5+a_4t^7}{(at+b)(1+a_1t+a_2t^2+a_3t^3)}+d_5\frac{t^6}{(at+b)(1+a_1t+a_2t^2+a_3t^3)}+t^8h(t)\\
&=\bigg(c_0b+(c_0(a+ba_1)+d_0)t+(c_0aa_1+c_0ba_2+d_0a_1)t^2+(c_0aa_2+c_0ba_3+d_0a_2)t^3\\
                    &\ \ \ \ \ +(c_0aa_3+d_0a_3+c_4b)t^4+(c_4a+c_5b+d_4)t^5\\
                     &\ \ \ \ \ +(c_4a_4b+c_5a+d_5)t^6+(c_4aa_4+d_4a_4)t^7\bigg) p(t)+t^8q(t)
\end{align*}
where $p(t)=1/((at+b)(1+a_1t+a_2t^2+a_3t^3))$, which has valuation $0$ at $P$. To reach the value $7$ we should kill all smaller coefficients. This implies $c_0=d_0=c_4=0$, $c_5b+d_4=0$ and $c_5a+d_5=0$. As the $c_i$ and $d_i$ are arbitrary, one clearly sees that $7\not\in D$, that is, $\deg(\aaa)=3$, if and only if $a_4=0$, and hence $C$ is trigonal; otherwise $\deg(\aaa)=4$ and $C$ is tetragonal since $\gon_F(C)\geq 4$ by Theorem \ref{thmggr}.(iii).}

\medskip
\emph{Now we claim that $\gon_F(C)=4$. First, note that
$$
f(t)=\frac{t^4}{1+a_1t+(a_2-a_4)t^2+(a_3-a_1a_4)t^3}=\frac{t^4+a_4t^6}{1+a_1t+a_2t^2+a_3t^3}+t^8h(t)\in\op
$$
with $h(t)\in \obp$. Indeed, consider 
\begin{align*}
   f_1(t)&=1+a_1t+(a_2-a_4)t^2+(a_3-a_1a_4)t^3 \\
f_2(t)&=1+a_1t+a_2t^2+a_3t^3.
\end{align*}
We are looking for $h(t)\in\obp$ for which
$$
\frac{t^4}{f_1(t)}=\frac{t^4+a_4t^6}{f_2(t)}+t^8h(t)
$$
Now
\begin{align*}
\frac{t^4}{f_1(t)}=\frac{t^4+a_4t^6}{f_2(t)}+t^8h(t)
 &\Leftrightarrow \frac{t^4}{f_1(t)}-\frac{t^4+a_4t^6}{f_2(t)}=t^8h(t)\\
    &\Leftrightarrow t^4f_2(t)-(t^4+a_4t^6)f_1(t)=t^8h(t)f_1(t)f_2(t) \\
    &\Leftrightarrow (a_4-a_2)a_4t^8+(a_1a_4-a_3)a_4t^9=t^8h(t)f_1(t)f_2(t) \\
    &\Leftrightarrow ((a_4-a_2)+(a_1a_4-a_3)t)a_4t^8=t^8h(t)f_1(t)f_2(t) 
\end{align*}
Thus take 
$$
h(t)=\dfrac{((a_4-a_2)+(a_1a_4-a_3)t)a_4}{f_1(t)f_2(t)}
$$
and note that, clearly, $h(t)\in\obp$. So $\aaa:=\oo_C\langle 1,f(t)\rangle$ has degree $0$ at $P$ since $f(t)\in\op$ and $4$ outside owing to (\ref{desi1}). Thus $\deg(\aaa)=4$ and $\aaa$ is locally free since $\aaa_P=\op$.}
\end{ex}

\medskip
To state the next result, fix the following notation: for ${\rm v}=(\rm v_1,\ldots,\rm v_r)\in\mathbb{N}^r$, set
$$
t^{\rm v}:=(t^{\rm v_1}:\cdots: t^{\rm v_r})\subset\mathbb{P}^{r-1}
$$

\begin{thm}
\label{general} Let $C$ be a rational curve with a unique singularity which is unibranch with semigroup $\sss^*=\{0,\alpha,\alpha+1,\ldots,\alpha+\ell-1,\alpha+\ell+m\}$ as depicted.
\begin{figure}[h]
\includegraphics[scale=0.4]{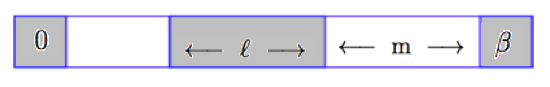}
\end{figure} 

Then $C$ is isomorphic to a curve of the form:
\begin{equation}
\label{equcv0}
C=\bigg(F_0: F_\alpha:  \cdots: F_{\alpha+\ell-2}:t^{\alpha+\ell-1}: t^{\rm v}\bigg)\subset \mathbb{P}^{N}
\end{equation}
where 
\begin{align*}
F_0&=1+a_1t+\cdots+a_{\ell+m-1}t^{\ell+m-1} \\
F_i&=t^i+\sum_{j=1}^{m-1}a_{\ell+(i-\alpha+1)(m-1)+j}t^{\alpha+\ell+j-1}\ \text{for}\ \alpha\leq i\leq\alpha+\ell-2\\
{\rm v}&=(\beta,\ldots, 2\alpha-1,2\alpha+2\ell-1, \ldots,\beta+\alpha-1)\ \text{and}\ N=\alpha-\ell+1\ \text{if}\ \ell\leq m,\\
{\rm v}&=(\beta,\ldots, 2\alpha-1)\ \text{and}\  N=\alpha-m\ \text{if}\ \ell>m.
\end{align*}

Moreover, the canonical model of $C$ given in terms of its coefficients is:

\begin{itemize}
\item [(i)] if $m=1$, then 
$$
C'=\bigg(1+\sum_{i=1}^{\ell}a_it^i: t^{\ell+1}: \cdots: t^{\alpha+\ell-1}\bigg)\subset \mathbb{P}^{g-1}.
$$

\item [(ii)] if $\ell=1$, then 
$$
C'=\bigg(1+\sum_{i=1}^{m}a_it^{i}:~h_1:~\cdots:~h_{m-1}:~t^{m+1}: \cdots: ~t^{\alpha+m-1}\bigg)\subset\mathbb{P}^{g-1}
$$ 
where
$$
h_i=\sum_{j=0}^{m-i}a_jt^{i+j}\ \text{for}\ 1\leq i\leq m-1.
$$

\item [(iii)] if $m=2$, then 
$$
C'=\bigg(1+\sum_{i=1}^{\ell+1}a_it^i:~t+\sum_{i=1}^{\ell}b_{2(\ell-i)+1}t^{i+1}:~h_{\ell+2}:\cdots:h_{\alpha+\ell}\bigg)\subset\mathbb{P}^{g-1}
$$ 
where
\begin{align*}
h_i &=\begin{cases}
t^i& ~~~~~\text{for}\ \ell+2\leq i\leq\alpha-1~\mbox{and}~i=\alpha+\ell-1,~\alpha+\ell \\
F_i&~~~~~\text{for}\ \alpha\leq i\leq\alpha+\ell-2
\end{cases}
\end{align*}
and
$$
b_{2i-1}=-\sum_{j=0}^{\ell-i}a_j(d_{\ell-i-j+1}+a_{\ell+i+j+1}).
$$
with $d_0=1$ e $d_i=-\sum_{j=1}^ia_jd_{i-j}$.
\end{itemize}

Besides, in the cases (i) and (ii) above $C$ is trigonal computed by a base point $g_3^1$, and this holds in (iii) as well if and only if 
$$
a_{\ell+i}=0\ \ \ \ \mbox{for}~2\leq i\leq \ell.
$$
Otherwise, $C$ is tetragonal computed by a base point $g_4^1$. Also, in all cases above, the smallest $d$ for which $C$ carries a base point free $g_d^1$ is $d=\alpha$.
\end{thm}

\begin{proof} We break down the proof of the above result into three natural parts.


\medskip
\noindent {\bf Part 1: Parametrization of the curve.} We will start by describing the singularity and then pass to a global parametrization. So let $P\in C$ be the unique singular point. Plainly, its local ring is  
\begin{equation*}
\label{equrng}
\op=k\oplus k f_1(t)\oplus \cdots \oplus k f_{\ell}(t)\oplus t^{\beta} \obp 
\end{equation*}
where $\beta=\alpha+\ell+m$ is the conductor number and $f_i(t) \in k(t)=k(C)$ are rational functions such that $\vpb(f_i)=\alpha+i-1$. This follows trivially from the fact that $\dim(\op/\cp)=\#(\sss\setminus(\beta+\mathbb{N}))$. On the other hand, we may consider the power series of each $f_i$ and just disregard its part on the conductor ideal. After standard normalization we get
\begin{align}
\label{equsi1}
f_1(t)&=t^\alpha+b_{1,1}t^{\alpha+\ell}+\cdots+b_{1,m}t^{\beta-1}\nonumber\\
\vdots\ \ \  &\ \ \ \ \ \ \ \ \ \ \ \ \  \ \ \ \ \ \ \ \ \ \ \vdots  \\
f_{\ell}(t)&=t^{\alpha+\ell-1}+b_{\ell,1}t^{\alpha+\ell}+\cdots+b_{\ell,m}t^{\beta-1} \nonumber .
\end{align}
So the number of free coefficients needed to describe the singularity matches the weight of the semigroup $w(\sss)=\sum_{i=1}^g \ell_i$, where the $\ell_i$ are the gaps. Here this number is $n:=\ell m$. 

\medskip
To get a global parametrization for $C$, we claim that any curve in the statement of the theorem is isomorphic to one like 
\begin{equation}
\label{equcv1}
C=(1:f_1:\ldots:f_{\ell}:t^{\rm v})
\end{equation}
where ${\rm v}$ is as in the statement of the theorem. Indeed, we must assure that the values of the parametrizing functions generate $\sss$. We have that the $f_i$ take care of $\sss\cap[0,\beta-1]$ but it remains to show that the set $[\beta,\beta+\alpha-1]$ is also reached by values of the parametrizing functions (for simplicity, we will omit intersection with $\nn$ in the notation of intervals). Now note that $2\alpha>\beta$. In fact, if $r$ is the largest integer such that $r\alpha <\beta$, then, clearly, $r\alpha\geq s':=\lfloor \beta/2 \rfloor$. But, by construction, $[\alpha,k\alpha]\subset\sss$. So if $r\geq 2$ then $s'\in\sss$ and so does $\beta-1=2s'$, which is a contradiction. Now $[2\alpha,2\alpha+2\ell-2]$ is reached by products $f_if_j$. So our choice of ${\rm v}$ is made to reach the remaining elements of $[\beta,\beta+\alpha-1]$. If $2\alpha+2\ell-2\geq \beta+\alpha-1$, that is, if $\ell>m$ (since $\beta=\alpha+\ell+m$) it suffices to take ${\rm v}=(\beta,\ldots, 2\alpha-1)$, otherwise we add the part $(2\alpha+2\ell-1, \ldots,\beta+\alpha-1)$ to ${\rm v}$.

\medskip
Now we make further modifications to the parametric functions $f_i$, replacing them with the ones in (\ref{equcv0}) in the theorem statement. The reasons why were explained in the digressions we made right before Theorem \ref{general}. So what we need to prove is that any choice of the $a_{k}$ as in (\ref{equcv0}) yields a curve like (\ref{equcv1}) and, conversely, the $b_{i,j}$ in (\ref{equcv1}) determine a curve like (\ref{equcv0}).

\medskip
First off, just note that the parameter space dimensions agree. In fact,  in (\ref{equcv0}) we have $(\ell+m-1)+(\ell-1)(m-1)=\ell m$ free coefficients like in (\ref{equsi1}) as desired. So set
\begin{equation}
\label{equsi2}
f_i^*:=\frac{F_{\alpha+i-1}}{F_0}=\frac{t^{\alpha+i-1}+\sum_{j=1}^{m-1}a_{\ell+i(m-1)+j}t^{\alpha+\ell+j-1}}{1+a_1t+\cdots+a_{\ell+m-1}t^{\ell+m-1}}\ \ \text{for}\ 1\leq i\leq \ell-1 
\end{equation}
and $f_{\ell}^*:=t^{\alpha+\ell-1}/F_0$. Write $1/F_0=\sum_{k=0}^{\infty}d_kt^k$ with $d_0=1$ and $d_j=-\sum_{k=1}^ja_kd_{j-k}$. Replacing it in (\ref{equsi2}) and develoving we get
\begin{align}
\label{equscp}
f_i^*&=\sum_{k=0}^{\ell-i}d_kt^{\alpha+i+k-1}+\sum_{r=0}^{m-2}\bigg(d_{\ell-i+r+1}+\sum_{k=0}^{r}d_ka_{\ell+i(m-1)+r-k+1}\bigg)t^{\alpha+\ell+r} \\
&\ \ \ \ \ \ \ \ \ \ \ \ \ \ \ \ \ \ \ \ \ \ \ \ \ +\bigg(d_{\ell+m-i}+\sum_{k=1}^{m-1}d_ka_{\ell+i(m-1)+m-k}\bigg)t^{\beta-1}+t^{\beta}h_i \nonumber
\end{align}
and $f_{\ell}^*=t^{\alpha+\ell-1}+\sum_{k=1}^md_kt^{\alpha+\ell+k-1}+t^{\beta}h_{\ell}$.

\medskip
In order to eliminate the powers which do not appear at (\ref{equsi1}), we now set
$$
\bar{f_{i}}:=f_i^*+\sum_{s=1}^{\ell-i}a_sf^*_{i+s}
$$
which yields
\begin{align*}
\bar{f_{i}}&=f_i^*+ \sum_{s=1}^{\ell-i}\sum_{k=0}^{\ell-i-s}a_sd_kt^{\alpha+i+s+k-1}\\ 
&\ \ \ \ \ \ \ \ +\sum_{r=0}^{m-2}\bigg(\sum_{s=1}^{\ell-i}a_s\bigg(d_{\ell-(i+s)+r+1}+\sum_{k=0}^{r}d_ka_{\ell+(i+s)(m-1)+r-k+1}\bigg)\bigg)t^{\alpha+\ell+r}\\ 
&\ \ \ \ \ \ \ \ \ \ +\bigg(\sum_{s=1}^{\ell-i}a_s\bigg(d_{\ell+m-(i+s)}+\sum_{k=1}^{m-1}d_ka_{\ell+(i+s)(m-1)+m-k}\bigg)\bigg)t^{\beta-1}+t^{\beta}g_i
\end{align*}
where $g_i\in\obp$. Now note that
$$
\sum_{k=0}^{\ell-i}d_kt^{\alpha+i+k-1}+ \sum_{s=1}^{\ell-i} \sum_{k=0}^{\ell-i-s}a_sd_kt^{\alpha+i+s+k-1}=t^{\alpha+i-1}
$$
which yields
\begin{align}
\label{equfbr}
\bar{f}_i&=t^{\alpha+i-1}+\sum_{r=0}^{m-2}\bigg(\sum_{s=0}^{\ell-i}a_s\bigg(d_{\ell-(i+s)+r+1}+\sum_{k=0}^{r}d_ka_{\ell+(i+s)(m-1)+r-k+1}\bigg)\bigg)t^{\alpha+\ell+r}\\
& \ \ \ \ \ \ \ \ \ \ \ \ \ +\bigg(\sum_{s=0}^{\ell-i}a_s\bigg(d_{\ell+m-(i+s)}+\sum_{k=1}^{m-1}d_ka_{\ell+(i+s)(m-1)+m-k}\bigg)\bigg)t^{\beta-1}+t^{\beta}h_i \nonumber
\end{align}
where $h_i\in\obp$. Thus, starting from the $f_i^*$, induced by (\ref{equcv0}), we just got the $\bar{f}_i\in\oo_P$ above, which are now in the same form of the $f_i$. Moreover, comparing (\ref{equsi1}) and (\ref{equfbr}), the $b_{i,j}$ are written in terms of the $a_k$ as follows
\begin{align*}
b_{i,j}&=\sum_{s=0}^{\ell-i}a_s\bigg(d_{\ell-i+j-s}+\sum_{k=0}^{j-1}d_ka_{\ell+(i+s)(m-1)+j-k}\bigg) ~~~~\text{for}\ 1\leq i\leq \ell-1, ~ 1\leq j\leq m-1\\
b_{i,m}&=\sum_{s=0}^{\ell-i}a_s\bigg(d_{\ell+m-i-s}+\sum_{k=1}^{m-1}d_ka_{\ell+(i+s)(m-1)+m-k}\bigg) ~~~~\text{for}\ 1\leq i\leq\ \ell-1 \\
b_{\ell,j}&~=d_j\ \ \ \ ~~~\text{for}\ 1\leq j\leq m.
\end{align*}

Conversely, to get the $a_k$ from the $b_{i,j}$, proceeding by induction on $j$ in the equations involving $b_{\ell,j}$ we get $a_j=-\sum_{k=1}^{j-1}a_kd_{j-k}-b_{\ell,j}$, for $1\leq j\leq m$. Similarly, but with a little more effort, for any fixed $i$, one may use induction on $j$ on the pair of equalities involving $b_{i,j}$ to get, for $1\leq j\leq m$,
\begin{align*}
a_{\ell+i(m-1)+j}&=b_{i,j}-d_{\ell-i+j}-\sum_{k=1}^{j-1}d_ka_{\ell+i(m-1)+j-k}\\
& \ \ \ \ \ \ \ \ \ \ \ \  -\sum_{s=1}^{\ell-i}a_s\bigg(d_{\ell-i+j-s}-\sum_{k=0}^{j-1}d_ka_{\ell+(i+s)(m-1)+j-k}\bigg)
\end{align*}
\begin{align*}
a_{\ell+m-i}&=-b_{i,m}-\sum_{k=1}^{\ell+m-i-1}a_kd_{\ell+m-i-k}+\sum_{k=1}^{m-1}d_ka_{\ell+i(m-1)+m-k} \\
& \ \ \ \ \ \ \ \ \ \ \ \  +\sum_{s=1}^{\ell-i}a_s\bigg(d_{\ell+m-i-s}+\sum_{k=1}^{m-1}d_ka_{\ell+(i+s)(m-1)+m-k} \bigg). 
\end{align*}
and we are done with the description of the curve.

\medskip
\noindent {\bf Part 2: Gonality.} Now we prove the last statements of the theorem concerning gonality, as such analysis, though related to, does not depend on the parametrization of the canonical model.

\medskip
\noindent {\bf (a) non-removable base point pencils.} We start by supposing the gonality is computed by a linear system that admits a non-removable base point, or equivalently, assuming that the sheaf $\aaa$ that computes the gonality is not locally free. 

\medskip
We begin with case (iii), where $m=2$. The analysis is essentially contained in Example \ref{exeexe}. Recall that $C$ cannot have gonality $2$, owing to \cite[Thm.~2.1]{Mt} since $\mmp\neq \cp$. So suppose $\deg(\aaa)=3$. Thus $\deg_P(\aaa)\geq 1$, as otherwise $\aaa_P=\oo_P$ and $\aaa$ would be a bundle. Then $\deg_{C\setminus P}(\aaa)\leq 2$.
Thus $r+\deg(f)\leq 2$ owing to (\ref{desi1}), and hence $r\leq 2$. If $r=2$ note that we have $\{2,\beta-1\}\in D$ since $m=2$, thus $\deg_P(\aaa)\geq 2$.  Therefore $\deg_{C\setminus P}(\aaa)\leq 1$ and all possibilities are precluded by (\ref{desi1}). 

\medskip
If $r=1$ then $\deg_P(\aaa)\geq 2$, so $\deg_{C\setminus P}(\aaa)\leq 1$. From (\ref{desi1}) we get $\deg(f)=0$, $\deg(h)\leq 1$ and $\deg_{C\setminus P}\mathcal{F}=1$. Therefore $\deg_P(\aaa)=2$. Now $\{1,\beta-2\}\subset D$ if $r=1$ and equality holds if and only if $\beta-1=\alpha+\ell+1\not\in D$. So we have to analyze when this property holds. 

\medskip
In order to check so, disregarding the functions with values greater than $\beta$, we have have that (\ref{equsum}) in the present case reduces to the direct direct sum
$$
\oo_P=k\oplus k \frac{F_{\alpha}(t)}{F_0(t)}\oplus \cdots \oplus k \frac{F_{\alpha+\ell-2}(t)}{F_0(t)}\oplus k \frac{t^{\alpha+\ell-1}}{F_0(t)}\oplus t^{\beta} \obp. 
$$
Write $\aaa=\oo_C\langle 1,t/(at+b)$ with $b\neq0$. If so, a general element in $\aaa_P$ is of the form
$$
f=\sum_{i\in S^{\circ}}\bigg(c_i+d_i\frac{t}{at+b}\bigg)\bigg(\frac{F_i}{F_0}\bigg)+t^{\beta} u
$$
where $u \in\obp$, $c_i,d_i\in\mathbb{C}$, and $\sss^{\circ}:=\{s\in\sss\,|\, s<\beta\}$. Setting $p(t):=1/(at+b)f_0$ and recalling that $m=2$ we may write
\begin{align*}
f=&\, p(t)\bigg[\,c_0b+\sum_{i=1}^{\ell+1}\bigg(c_0(aa_{i-1}+ba_{i})+a_{i-1}c_1\bigg)t^i\\
&\ \ \ \ \ \ \ \ +a_{\ell+1}(ac_0+c_1)t^{\ell+2}+bc_{\alpha}t^\alpha+\sum_{i=\alpha+1}^{\alpha+\ell-1}\bigg(d_{i-1}+bc_i+ac_{i-1}\bigg)t^i\\
&\ \ \ \ \ \ \ \ +\bigg(d_{\alpha+\ell-1}+ac_{\alpha+\ell-1}+\sum_{i=\alpha}^{\alpha+\ell-2}\bigg(ba_{\ell+i-\alpha+2}c_i\bigg)\bigg)t^{\alpha+\ell}\\
&\ \ \ \ \ \ \ \ +\bigg(\sum_{i=\alpha}^{\alpha+\ell-2}a_{\ell+i-\alpha+2}(d_{i}+ac_i)\bigg)t^{\alpha+\ell+1}\bigg] +t^{\beta} q(t)
\end{align*}
Now assume $\vpb(f)=\alpha+\ell+1$. Since $\vpb(p)=0$, we should kill all coefficients above but the last one. From the vanishing of all coefficients until $t^{\alpha+\ell-1}$ we get $c_0=c_1=c_\alpha=0$,  $d_{\alpha}=-bc_{\alpha+1}$ and
\begin{equation}
\label{equca11}
d_{\alpha+i-1}=-bc_{\alpha+i}-ac_{\alpha+i-1}\ \ \ \ \  \text{for}\ 2\leq i\leq \ell-1.
\end{equation}
Killing $t^{\alpha+\ell}$ yields
\begin{equation}
\label{equca12}
d_{\alpha+\ell-1}=-ac_{\alpha+\ell-1}-\sum_{i=\alpha}^{\alpha+\ell-2}ba_{\ell+i-\alpha+2}c_i.
\end{equation}
So now 
$$
f=\bigg(\sum_{i=\alpha}^{\alpha+\ell-2}a_{\ell+i-\alpha+2}(d_{i}+ac_i)\bigg)t^{\alpha+\ell+1}+\cdots = \bigg(-\sum_{i=\alpha+1}^{\alpha+\ell-1}ba_{\ell+i-\alpha+1}c_i\bigg)t^{\alpha+\ell+1}+\cdots
$$
where the second equality comes from (\ref{equca11}) and (\ref{equca12}). As $b\neq0$ and the $c_i$ are arbitrary, we see that $\alpha+\ell+1\not\in D$ if and only if
\begin{equation}
\label{equtrg}
a_{\ell+i-\alpha+1}=0\ \ \ \ \mbox{for}~\alpha+1\leq i\leq \alpha+\ell-1.
\end{equation}
So just in this case $\aaa$ has degree $3$ and $C$ trigonal computed by a base point $g_3^1$. Othewise $C$ is tetragonal computed by the same sheaf, unless $\alpha=3$, as we will see right below, since in this case $C$ carries a base point free $g_3^1$.

\medskip
Now we address the gonality in cases (i) and  (ii). Form the sheaf $\aaa:=\oo_C \langle 1, t\rangle$.  Clearly, $\{1,\alpha+\ell\}\subset D$.  We claim the inclusion is an equality in both cases. Indeed, this is immediate if $m=1$. On the other hand, if $\ell=1$, write a general element of $\aaa_P$ as
$$
f=c_0+c_1t+c_{2}\frac{f_{\alpha}}{f_0}+c_{3}\frac{tf_{\alpha}}{f_0}+t^{\beta}h
$$
where $h\in\obp$. If we also write $f_{\alpha}/f_0 = t^{\alpha}+dt^{\alpha+1}+\cdots$, we get
\begin{align*}
f&=c_0+c_1t+c_2 t^\alpha+(c_2d+c_3)t^{\alpha+1}+\cdots
\end{align*}
So if $\vpb(f)\geq \alpha+2$, then $c_0=c_1=c_2=c_3=0$, and hence $\vpb(f)\geq\beta$. Thus $D=\{1,\alpha+1\}$ as desired. Therefore, in both cases, $\deg_P(\aaa)=\#D=2$. Now $\mathcal{F}$ has degree 1 at $Q$ and 0 at the open set $U:=C\setminus\{P,Q\}$, so its total degree is $3$. Thus $C$ is trigonal computed by a base point $g_3^1$ if conditions of (i) or (ii) apply.

\medskip
\noindent {\bf (b) base point free pencils.} We now study the gonality concerning just base point free pencils, or equivalently, assuming the linear system is induced by a locally free sheaf. We claim that any $C$ as stated in the theorem, always carries a base point free $g_{\alpha}^1$. In fact, set
$$
f(t):=\frac{t^{\alpha}}{h(t)}:=\frac{t^{\alpha}}{b_0+b_1t+\cdots+b_{\ell+m-1}t^{\ell+m-1}}
$$
with
\begin{align*}
b_i &=\begin{cases}
1&~~~~~\text{for}\ i=0\\
a_i&~~~~~\text{for}\ 1\leq i\leq \ell-1\\
a_i-\sum_{j=\ell}^ia_{m+j}b_{i-j}&~~~~~\text{for}\ \ell\leq i\leq \ell+m-2\\
a_i-\sum_{j=1}^{m-1}a_{j+\ell+m-1}b_{m-j}&~~~~~\text{for}\ i=\ell+m-1.
\end{cases}
\end{align*}
One can check that
$$
f(t)=\frac{t^{\alpha}+\sum_{j=1}^{m-1}a_{j+\ell+m-1}t^{\alpha+j+\ell-1}}{1+a_1t+\cdots+a_{\ell+m-1}t^{\ell+m-1}}+t^{\beta}q(t)=\frac{F_{\alpha}}{F_0} +t^{\beta}q(t) \in \oo_P
$$
where $q(t)\in\obp$. Now set $\aaa:=\oo_C\langle 1,f(t)\rangle$. Note that $\deg_P(\aaa)=0$ as $f(t)\in \oo_P$.
On the other hand, $\ell+m-1\leq \alpha$, otherwise one finds $s$ and $\beta-s-1$ both in $\sss$ which never happens. Hence $\deg(h(t))\leq \alpha$. Using (\ref{desi1}) we get
$$
\deg_{C\setminus P}=
\deg(h(t))+{\rm max}\{0,\alpha -\deg(h(t))\}=\alpha
$$
By construction, $\aaa$ is locally free, thus induces a base point free $g_\alpha^1$ on $C$. Theorem \ref{thmggr}.(iii) then assures $\gon_F(C)=\alpha.$

\medskip

\medskip
\noindent {\bf Part 3: Canonical Model.} Now we compute the canonical model of $C$.
First note that $\ww$ is characterized by being a torsion free sheaf of rank $1$ and degree $2g-2$ with at least $g$ independent global sections. Indeed, if $\aaa$ is any such a sheaf, then $\chi(\aaa)=g-1$ since $\deg(\aaa)=2g-2$, and hence $h^1(\aaa)\geq 1$ since $h^0(\aaa)\geq g$. But the nonvanishing of $H^1(\aaa)$ yields an inclusion $\aaa\hookrightarrow\ww$ and these sheaves must coincide as they are of same degree. 

\medskip
To address case (i), and based on Theorem \ref{thmggr}.(iv), if $m=1$ we get
$$
{\rm K}^{\circ}=\{0,\ell+1,\ldots,\alpha+\ell-1\}.
$$ 
So consider the vector space
$$
V:=\bigg\langle 1, \frac{t^{\ell+1}}{h_0},\ldots,\frac{t^{\ell+\alpha-1}}{h_0}\bigg\rangle \subset k(t)
$$
where
$$
h_0  =1+a_1t+\cdots+a_{\ell}t^{\ell}.
$$
and take $\aaa:=\oo_C\langle V\rangle$. If we show that $\deg(\aaa)=2g-2$ we are done. Indeed,  $h^0(\aaa)\geq\dim(V)=\#\kk^{\circ}=g$, so if $\deg(\aaa)=2g-2$ then $\aaa=\ww$ and $V=H^0(\ww)$.

\medskip
Since $\ell <\alpha$, we have $\deg_Q(\aaa)=\alpha+\ell-1-\deg(h_0)$. On the other hand, $\deg_U(\aaa)=\deg(h_0)$ where $U:=C\setminus\{P,Q\}$. Therefore, $\deg_{C\setminus\{P\}}=\alpha+\ell-1$. Finally, we compute the degree at $P$. For simplicity, given $T\subset\nn$ set $T':=T\cap(0,\beta)$. Write 
$$
\op=k\oplus \bigg(\bigoplus _{i\in \sss'} k\frac{t^i}{h_0}\bigg)\oplus \cp
$$
Since $\aaa_P=\oo_P\langle V\rangle$ and $i+j\geq\beta$ for every $i\in\sss'$ and $j\in \kk'$ we get
$$
\aaa_P=k\oplus\bigg(\bigoplus_{i\in \kk'} k\,\frac{t^i}{h_0}\bigg)\oplus\cp
$$
Therefore $\deg_P(\aaa)=\#(\kk\setminus\sss)=\alpha-\ell-1$. Summing up we get 
$$
\deg(\aaa)=(\alpha-\ell-1)+(\alpha+\ell-1)=2\alpha-2
$$
Now $g$ equals the number of gaps of $\sss$, which is $\alpha$. Hence $\deg(\aaa)=2g-2$ as claimed, and $V=H^0(\ww)$. We then get $C'$ as stated in (i).

\medskip
For the case (ii), if $\ell=1$ we get
$$
{\rm K}^{\circ}=\{0,1,\ldots,m-1,m+1,\ldots,\alpha+m-1\}.
$$ 
So consider the functions
$$
h_i = 
\begin{cases}
\sum_{j=0}^{m-i}a_jt^{i+j} &\text{for}\ 0\leq i\leq m-1 \\
t^i &\text{for}\ m+1\leq i\leq \alpha+m-1
\end{cases}
$$
and the sheaf
$$
\aaa:=\oo_C\bigg\langle 1,\frac{h_{1}}{h_0},\ldots,\frac{h_{m-1}}{h_0},\frac{t^{m+1}}{h_0},\ldots,\frac{t^{\alpha+m-1}}{h_0}\bigg\rangle.
$$
As before, we will show that $\deg(\aaa)=2g-2$. By similar arguments to the prior case we get that $\deg_{C\setminus\{ P\}}(\aaa)=\alpha+m-1$. To compute the degree at $P$ write
\begin{equation}
\label{equm2}
\aaa_P
=\bigg(\bigoplus_{i\in \kk\cap[0,\alpha]} k\frac{h_i}{h_0}\bigg)\oplus\bigg(\sum_{i=\alpha+1}^{\alpha+m-1} k\frac{h_i}{h_0}+\sum_{i=1}^{m-1} k\frac{h_ih_\alpha}{h_0^2}+\cp\bigg)
\end{equation}
Now recall that if we set $D:=\vpb(\aaa_P)\setminus\sss$ then local degree is $\deg_P(\aaa)=\#D$. Set $E:=\kk\setminus\sss=\{1,\ldots,\alpha+m-1\}\setminus\{m,\alpha\}$. The decomposition above yields
$$
E\subset D \subset E\cup\{\alpha+m\}.
$$
We claim that $\alpha+m\not\in D$. Indeed, write an element of the second summand above with order smaller than $\beta$ as
\begin{equation}
\label{equm2}
f=\sum_{i=\alpha+1}^{\alpha+m-1} c_i\frac{t^i}{h_0}+\sum_{i=1}^{m-1} b_i\frac{h_ih_\alpha}{h_0^2}
\end{equation}
and also $1/h_0^2=\sum_{j=0}^{\infty}p_jt^j$ with $p_0=1$ and $p_j=d_j-\sum_{k=1}^ja_kp_{j-k}$. We get
\begin{align*}
f&=\sum_{i=\alpha+1}^{\alpha+m-1}\bigg(\sum_{j=0}^{i-\alpha-1}\bigg(d_jc_{i-j}
    +\sum_{k=0}^ja_kp_{j-k}b_{i-j-\alpha}\bigg)\bigg)t^i\\
    &\ \ \ \ \ \ \ +\bigg(\sum_{i=1}^{m-1}\bigg(d_ic_{\alpha+m-i}+\sum_{j=0}^ia_jp_{i-j}b_{m-i}\bigg)\bigg)t^{\alpha+m}+t^{\beta}h
\end{align*}
with  $h\in\obp$. Killing the coefficients of $t^{\alpha+i}$ we see that
\begin{equation}
\label{eqm3}
c_{\alpha+i}=-b_i\ \ \ \ \  \text{for}\ 1\leq i\leq m-1.
\end{equation}
Using (\ref{eqm3}) we have that the coefficient of $t^{\alpha+m}$ vanishes as well. This implies that $\alpha+m\not\in D$. Then, $E=D$ and so $\deg_P(\aaa)=\alpha+m-3$ and hence 
$$
\deg(\aaa)=(\alpha+m-3)+(\alpha+m-1)=2(\alpha+m-1)-2.
$$
Now $g$, the number of gaps of $\sss$, is $\alpha+m-1$. Hence $\deg(\aaa)=2g-2$ and $C'$ is as stated in (ii).

\medskip
And for the case (iii), if $m=2$ we get 
$$
{\rm K}^{\circ}=\{ 0,1,\ell+2,\ldots,\alpha+\ell\}.
$$ 
So take the functions
\begin{align*}
h_0 & =1+a_1t+\cdots+a_{\ell+1}t^{\ell+1} \\
h_1 &=t+b_{n-1}t^2+b_{n-3}t^3+\cdots+b_1t^{\ell+1}\\
h_i &=\begin{cases}
t^i& ~~~~~\text{for}\ \ell+2\leq i\leq\alpha-1~\mbox{and}~i=\alpha+\ell-1,~\alpha+\ell \\
F_i&~~~~~\text{for}\ \alpha\leq i\leq\alpha+\ell-2
\end{cases}
\end{align*}
where 
$$
b_{2i-1}=-\sum_{j=0}^{\ell-i}a_j\bigg(d_{\ell-i-j+1}+a_{\ell+i+j+1}\bigg)
$$ 
and consider the sheaf
$$
\aaa:=\oo_C\bigg\langle 1,\frac{h_{1}}{h_0},\frac{t^{\ell+2}}{h_0},\ldots,\frac{t^{\alpha-1}}{h_0},\frac{F_{\alpha}}{h_0},\ldots,\frac{F_{\alpha+\ell-2}}{h_0},\frac{t^{\alpha+\ell-1}}{h_0},\frac{t^{\alpha+\ell}}{h_0}\bigg\rangle.
$$
Again, we will show that $\deg(\aaa)=2g-2$. We have that $\deg_{C\setminus\{ P\}}(\aaa)=\alpha+\ell$ and we claim that $\deg_P(\aaa)=\alpha-\ell$. To compute the degree at $P$ write
\begin{equation}
\label{equm22}
\aaa_P
=\bigg(\bigoplus_{i\in \kk\cap[0,\alpha]} k\frac{h_i}{h_0}\bigg)\oplus\bigg(\sum_{i=\alpha+1}^{\alpha+\ell} k\frac{h_i}{h_0}+\sum_{i=\alpha}^{\alpha+\ell-1}k\frac{h_1h_i}{h_0^2}\bigg)\oplus\cp
\end{equation}
Here $E=\{1\}\cup\{\ell+2,\ldots,\alpha-1\}\cup\{\alpha+\ell\}$. The decomposition (\ref{equm22}) yields
$$
E\subset D \subset E\cup\{\alpha+\ell+1\}.
$$

We claim that $\alpha+\ell+1 \not\in D$. Indeed, write an element of the second summand above with order smaller than $\beta$ as
\begin{equation}
\label{equm2}
f=\sum_{i=\alpha+1}^{\alpha+\ell} c_i\frac{h_i}{h_0}+\sum_{i=\alpha}^{\alpha+\ell-1} b_i\frac{h_1h_i}{h_0^2}
\end{equation}
We get
\begin{align*}
f&=\sum_{i=\alpha+1}^{\alpha+\ell-1}\bigg(\sum_{j=0}^{i-\alpha-1}\bigg(d_jc_{i-j}+B_jb_{i-j-1}\bigg)\bigg)t^i
+\sum_{j=0}^{\ell-1}\bigg(B_jb_{\alpha+\ell-j-1}+(d_j+a_{2(\ell+1)-j})c_{\alpha+\ell-j}\bigg)t^{\alpha+\ell}\\
& \ \ \ \ \ \ \ +\bigg(\sum_{j=0}^{\ell-1}\bigg(B_{j+1}+a_{2\ell-j+1}\bigg)b_{\alpha+\ell-j-1}+\bigg(d_{j+1}+d_1a_{2(\ell+1)-j}\bigg)c_{\alpha+\ell-j}\bigg)t^{\alpha+\ell+1}+t^{\beta}h
\end{align*}
with $h\in\obp$, $B_i=\sum_{j=0}^ip_jb_{2(\ell-i+j)+1}~\mbox{e}~b_{2\ell+1}=1.$ 

\medskip
Now note that for $f$ to have order $\alpha+\ell+1$ we should have first
\begin{equation}
\label{equca2}
c_{\alpha+1}=-b_{\alpha}.
\end{equation}
Keeping up with this procedure, we have
\begin{equation}
\label{equca3}
c_{\alpha+i}=-b_{\alpha+i-1}+\sum_{j=1}^{i-2}a_{2\ell-j+1}b_{\alpha+i-j-2}\ \ \ \ \  \text{for}\ 2\leq i\leq \ell-1.
\end{equation}
Assuming the vanishing of the coefficient of $t^{\alpha+\ell}$ we get
\begin{align*}
c_{\alpha+\ell}&=-b_{\alpha+\ell-1}-\sum_{j=1}^{\ell-1}\bigg(B_jb_{\alpha+\ell-j-1}+(d_j+a_{2(\ell+1)-j})c_{\alpha+\ell-j}\bigg).
\end{align*}
From (\ref{equca2}) and (\ref{equca3}) we have
$$
\sum_{j=1}^{\ell-1}\bigg(B_jb_{\alpha+\ell-j-1}+d_jc_{\alpha+\ell-j}\bigg)=-\sum_{j=2}^{\ell-1}\bigg(a_{2(\ell+1)-j}b_{\alpha+\ell-j-1}\bigg)
$$
which yields
\begin{equation}
\label{equca4}
c_{\alpha+\ell}=-b_{\alpha+\ell-1}+\sum_{j=2}^{\ell-1}\bigg(a_{2(\ell+1)-j}b_{\alpha+\ell-j-1}-a_{2(\ell+1)-j}c_{\alpha+\ell-j}\bigg).
\end{equation}
And now assuming the vanishing of the coefficient of $t^{\alpha+\ell+1}$ we have
\begin{align*}
d_1c_{\alpha+\ell}&=-B_1b_{\alpha+\ell-1}-\sum_{j=1}^{\ell-1}\bigg(\bigg(B_{j+1}+a_{2\ell-j+1}\bigg)b_{\alpha+\ell-j-1}+\bigg(d_{j+1}+d_1a_{2(\ell+1)-j}\bigg)c_{\alpha+\ell-j}\bigg).
\end{align*}
Note that
$$
\sum_{j=1}^{\ell-1}\bigg(B_{j+1}b_{\alpha+\ell-j-1}+d_{j+1}c_{\alpha+\ell-j}\bigg)=-\sum_{j=1}^{\ell-1}\bigg(a_{2\ell-j+1}b_{\alpha+\ell-j-1}+d_1a_{2(\ell+1)-j}b_{\alpha+\ell-j-1}\bigg)
$$
which yields
\begin{align}
\label{equca5}
d_1c_{\alpha+\ell}&=-B_1b_{\alpha+\ell-1}+\sum_{j=2}^{\ell-1}d_1\bigg(a_{2(\ell+1)-j}b_{\alpha+\ell-j-1}-a_{2(\ell+1)-j}c_{\alpha+\ell-j}\bigg).
\end{align}
Finally, from (\ref{equca4}) and (\ref{equca5}) we deduce that the coefficient of $t^{\alpha+\ell+1}$ vanishes. This implies that $\alpha+\ell+1\not\in D$. Therefore $E=D$ and hence $\deg_P(\aaa)=\alpha-\ell$. Thus
$$
\deg(\aaa)=(\alpha-\ell)+(\alpha+\ell)=2\alpha.
$$
But as the number of gaps of $\sss$ is $\alpha+1$, we get $\deg(\aaa)=2g-2$, and $C'$ as desired. Theorem \ref{general} is then proved. 
\end{proof}

\medskip
In the sequence, we state a similar result to Theorem \ref{general}, studying unicuspidal rational curves of which the singularity has two blocks of positive values smaller than the conductor number, the first one being single. The proof goes in a very close way, following the same steps based in Theorem \ref{thmggr}, and can be found in \cite{GNM}.

\pagebreak
\begin{thm}
\label{general1} Let $C$ be a rational curve with a unique singularity which is unbranched with semigroup $\sss^*=\{0,\alpha,\alpha+m+1,\ldots,\alpha+m+\ell-1,\alpha+m+\ell+1 \}$ as depicted
\begin{figure}[h]
\centering
\includegraphics[scale=0.5]{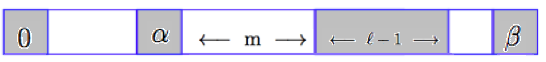}
\end{figure}

Then $C$ is isomorphic to a curve of the form: 
\begin{equation}
\label{equcv2}
C=\bigg(F_0: F_\alpha: t^{\alpha+m+1}:\cdots :t^{\beta-2}: t^{\rm v}\bigg)\subset \mathbb{P}^{N}
\end{equation}
where
$$
F_0=1+\sum_{i=1}^ua_it^i\ \ \ \ \text{and}\ \ \ \ F_{\alpha}=t^\alpha+\sum_{i=1}^va_{u+i}t^{\alpha+i}
$$
or
$$
F_0=1+\sum_{i=1}^{u-1}a_it^i+a_ut^n\ \ \ \ \text{and}\ \ \ \  F_{\alpha}=t^\alpha+\sum_{i=1}^{v-1}a_{u+i}t^{\alpha+i}+\bigg(a_v+\sum_{j=1}^{v-1}\bigg(a_jd_{v-j}-d_ja_{n-j}\bigg)\bigg)t^{\alpha+v}
$$
with
\begin{align*}
u:&=\max\{\ell,m+1\},~v:=\min\{\ell-1,m\},~n:=\ell+m, \\
{\rm v}:&=(\beta,\ldots,2\alpha+m,2\alpha+n,\ldots,2(\alpha+m)+1)\ \text{and}\ N=\alpha+m-\ell+2\ \text{if}\ \ell\leq m+1,\\
{\rm v}:&=(\beta,\ldots,2\alpha+m)\ \text{and}\  N=\alpha\ \text{if}\ \ell\geq m+2.
\end{align*}

Moreover, the canonical model of $C$ given in terms of its coefficients is:
$$
C'=\bigg(h_0:h_{\ell}:\cdots:~h_{n-1}:~h_{n+1}:~h_{n+2}: \cdots:~h_{\alpha+n-1}\bigg) \subset \mathbb{P}^{g-1}
$$ 
where
\begin{align*}
h_i &=\begin{cases}
t^i+\sum_{j=1}^{n-i}e_jt^{i+j}&~~~~~\text{for}\ i\in \{\ell,\ldots,n-1\}\setminus \{\alpha\}\\
t^i&~~~~~\text{for}\ i\in \{n+1,\ldots,\beta-2\}\setminus \{\alpha\}
\end{cases}
\end{align*}
with 
$$
e_i=a_i-\sum_{j=1}^va_{u+j}e_{i-j}, \ \ \ \ \ \ e_0=1\ \ \ \ \ \ \ \mbox{and} \ \ \ \ \ \ e_j=0~\mbox{if}~j<0
$$
and, either
$$
h_0=1+\sum_{i=1}^ua_it^i\ \ \ \ \text{and}\ \ \ \ h_{\alpha}=t^\alpha+\sum_{i=1}^va_{u+i}t^{\alpha+i}
$$
or
$$
h_0=1+\sum_{i=1}^{u-1}a_it^i+a_ut^n\ \ \ \ \text{and}\ \ \ \  h_{\alpha}=t^\alpha+\sum_{i=1}^{v-1}a_{u+i}t^{\alpha+i}+\bigg(a_v+\sum_{j=1}^{v-1}\bigg(a_jd_{v-j}-d_ja_{n-j}\bigg)\bigg)t^{\alpha+v}
$$
depending on cases which will be addressed.

\medskip
Also, $C$ is trigonal computed by a base point $g_3^1$ if and only if is of the form  $\sss^*=\{0,\alpha,\alpha+2,\alpha +4\}$. Otherwise, $C$ is tetragonal computed by a base point $g_4^1$. Moreover, the smallest $d$ for which $C$ carries a base point-free $g_d^1$ is $d=\alpha$.

\end{thm}

\section{Rational Unicuspidal Curves of Genus at most $6$}

\label{cap5}

In this section we apply Theorems \ref{general} and \ref{general1} to classify all unicuspidal rational non-Gorenstein curves of genus $g\leq 6$, by means of different gonalities and their canonical models. It is stated in the following theorem-like tableau.

\begin{thm}
\label{thmgn6}
Let $C$ be a unicuspidal rational non-Gorenstein curve of genus $g\leq 6$ with canonical model $C'$. For short, set $d_b:=\gon_B(C)$ and $d_f=\gon_F(C)$. Then we have the following possibilities for $C$ up to isomorphism.
\end{thm}

\begin{center}
$ $
\begin{tabular}{|c|c|c|c|c|}
\hline
{\rm case} & $C$ {\rm and} $C'$ & $d_{b}$ & $d_f$ \\ 
\hline 
& {\rm genus 3}  & & \\ \hline
{\rm (i)} &  $(1+a_1t\, :\, t^3 :\, t^5\, :\, t^6:\, t^7)$ & $3$ & $3$ \\
& $(1+a_1t\, :\, t^2\,:\,t^3)$ & & \\ \hline
& {\rm genus 4}  & & \\ \hline
{\rm (ii)} &  $(1+a_1t+a_2t^2\, :\, t^3 :\, t^7\, :\, t^8)$ & $3$ & $3$ \\
& $(1+a_1t+a_2t^2\, :\, t+a_1t^2\, :\, t^3\,:\,t^4)$ & & \\ \hline
{\rm (iii)} &  $(1+a_1t+a_2t^2\, :\, t^4 :\, t^5\, :\, t^7\, :\, t^8)$ & $3$ & $4$ \\
& $(1+a_1t+a_2t^2\, :\, t^3\, :\, t^4\,:\,t^5)$ & & \\ \hline
{\rm (iv)} &  $(1+a_1t\, :\, t^4 :\, t^6\, :\, t^7\, :\, t^9\, :\, t^{10})$ & $3$ & $4$ \\
& $(1+a_1t\, :\, t^2\, :\, t^3\,:\,t^4)$ & & \\ \hline
& {\rm genus 5}  & & \\ \hline
&  $(1+a_1t+a_2t^2\, :\, t^4+a_3t^5 :\, t^6\, :\, t^9 :\,t^{10}\,: t^{11})$ &  &  \\
 & $(1+a_1t+a_2t^2\, :\, t^2+(a_1-a_3)t^3\,:\,t^4+a_3t^5\, :\,t^5 :\,t^6)$ & & \\ 
{\rm (v)} & or & $3$ & $4$\\
&  $(1+a_1t+a_2t^3\, :\, t^4+a_1t^5 :\, t^6\, :\, t^9 :\,t^{10}\,: t^{11})$ &  &  \\
& $(1+a_1t+a_2t^3\, :\, t^2\,:\,t^4+a_1t^5\, :\,t^5 :\,t^6)$ &  &  \\ \hline 
{\rm (vi)}& $(1+a_1t+a_2t^2+a_3t^3\, :\, t^4+a_4t^6 :\, t^5\, :\,t^{10} :\,t^{11})$   & 3 & 4 \\ 
& $(1+a_1t+a_2t^2+a_3t^3\, :\, t+a_1t^2+(a_2-a_4)t^3\,:\,t^4+a_4t^6\, :\,t^5: \,t^6)$  & & \\ \hline
{\rm (vii)} & $(1+a_1t+a_2t^2\, :\, t^5 :\, t^6\, :\, t^8\, :\,t^9)$  & 3 & 5 \\ 
& $(1+a_1t+a_2t^2\, :\, t^3\,:\,t^4\, :\,t^5: \,t^6)$  & & \\ \hline
{\rm (viii)} & $(1+a_1t+a_2t^2+a_3t^3\, :\, t^5 :\, t^6\, :\, t^7\, :\,t^9  :\,t^{10})$ & 3 & 5 \\
& $(1+a_1t+a_2t^2+a_3t^3\, :\, t^4\,:\,t^5\, :\,t^6: \,t^7)$  & & \\ \hline
{\rm (ix)} & $(1+a_1t\, :\, t^5 :\, t^7\, :\, t^8 :\,t^9 :\, t^{10} :\,t^{11})$  & 3 & 5 \\ 
& $(1+a_1t\, :\, t^2\, :\,t^3 :\,t^4 :\,t^5)$  & & \\ \hline
{\rm (x)}  & $1+a_1t+a_2t^2\, :\, t^4 :\, t^7\, :\,t^9 :\,t^{10})$  & 3 & 4 \\ 
& $(1+a_1t+a_2t^2\, :\, t+a_1t^2\, :\,t^3 :\,t^4 :\,t^5)$  & & \\\hline
& $(1+2a_4t+a_2t^2+a_3t^3\, :\, t^3+a_4t^4 :\, t^8 :\,t^{10}\,: t^{11})$ & & \\
& $(1+2a_4t+a_2t^2+a_3t^3\, :\, t^2+a_4t^3+(a_2-a_4^2)t^4\,:\,t^3+a_4t^4\, :\,t^5 :\,t^6)$ &  &  \\
{\rm (xi)} & or & 4 & 3\\
& $(1+a_2t^2+a_3t^4\, :\, t^3 :\, t^8 :\,t^{10}\,: t^{11})$ & & \\
& $(1+a_2t^2+a_3t^4\, :\, t^2+a_2t^4\,:\,t^3\, :\,t^5 :\,t^6)$ &  &  \\ \hline
&  $(1+a_1t+a_2t^2+a_3t^3\, :\, t^3+a_4t^4+a_5t^5\, : \, t^7 :\,t^{10} :\,t^{11})$ & & \\
& $(1+a_1t+a_2t^2+a_3t^3 :\,t^3+a_4t^4+a_5t^5 :\,t^4+(a_1-a_4)t^5\, :\,t^6 :\,t^7)$ &  &  \\
{\rm (xii)} & or & 4 & 3\\
&  $(1+a_1t+(a_1^2-a_4^2)t^2+a_3t^5\, :\, t^3+a_4t^4+a_5^*t^5\, : \, t^7 :\,t^{10} :\,t^{11})$ & & \\
& $(1+a_1t+(a_1^2-a_4^2)t^2+a_3t^5 :\,t^3+a_4t^4+a_5^*t^5 :\,t^4+(a_1-a_4)t^5\, :\,t^6 :\,t^7)$ &  &  \\ 
& where $a_2=((a_1-a_4)^2+2a_5)$ and $a_5^*=a_4(a_1-a_4)$  & & \\ \hline
\end{tabular}
\end{center}

\begin{center}
$ $
\begin{tabular}{|c|c|c|c|c|}
\hline
& {\rm genus 6}  & & \\ \hline
{\rm (xiii)} & $(1+a_1t\, :\, t^6 \, :\, t^8\, :\,t^9  :\,t^{10}  :\,t^{11}  :\,t^{12} :\,t^{13})$  & 3 & 6 \\ 
&   $(1+a_1t\, :\, t^2\,:\,t^3\, :\,t^4  :\,t^5  :\,t^6)$ & & \\ \hline
{\rm (xiv)} & $(1+a_1t+a_2t^2\, :\, t^6\, :\, t^7\, : \, t^9 :\,t^{10} :\,t^{11})$  & 3 & 6 \\ 
&  $(1+a_1t+a_2t^2\, :\, t^3\,:\,t^4\, :\,t^5 :\,t^6\, :\,t^7)$ & & \\ \hline
{\rm (xv)}& $(1+a_1t+a_2t^2+a_3t^3\, :\, t^6 :\, t^7\, :\, t^8\, :\,t^{10}  :\,t^{11})$  & 3 & 6 \\ 
& $(1+a_1t+a_2t^2+a_3t^3\, :\, t^4\,:\,t^5\, :\,t^6: \,t^7: \,t^8)$ & & \\ \hline
{\rm (xvi)} & $(1+a_1t+a_2t^2+a_3t^3+a_4t^4\, :\, t^6 :\, t^7\, :\, t^8\, :\,t^9  :\,t^{11} :\,t^{12})$  & 3 & 6 \\ 
&  $(1+a_1t+a_2t^2+a_3t^3+a_4t^4\, :\, t^5\,:\,t^6\, :\,t^7 :\,t^8 :\,t^9)$ & & \\ \hline
{\rm (xvii)}&  $(1+a_1t+a_2t^2\, :\, t^5 :\, t^8\, :\,t^9 :\,t^{11} :\,t^{12})$  & 3 & 5 \\ 
&   $(1+a_1t+a_2t^2\, :\, t+a_1t^2\, :\,t^3 :\,t^4 :\,t^5 :\,t^6)$ & & \\ \hline
&  $(1+a_1t+a_2t^2\, :\, t^5+a_3t^6 :\, t^7\, :\, t^9 :\,t^{11}\,: t^{12}\,: t^{13})$ &  &  \\
 & $(1+a_1t+a_2t^2\, :\, t^2+(a_1-a_3)t^3\,:\,t^4\, :\,t^5+a_3t^6 :\,t^6\,: t^7)$ & & \\ 
{\rm (xviii)} & or & 3 & 5\\
&  $(1+a_1t+a_2t^3\, :\,  t^5+a_3t^6 :\, t^7\, :\, t^9 :\,t^{11}\,: t^{12}\,: t^{13})$ &  &  \\
& $(1+a_1t+a_2t^3\, :\, t^2+(a_1-a_3)t^3\,:\,t^4\, :\,t^5+a_3t^6 :\,t^6\,: t^7)$ &  &  \\ \hline 
&  $(1+a_1t+a_2t^2+a_3t^3\, :\, t^5+a_4t^6 :\, t^7\, : \, t^8 :\,t^{10} :\,t^{11})$ & & \\
& $(1+a_1t+a_2t^2+a_3t^3\, :\, t^3+(a_1-a_4)t^4\,:\,t^5+a_4t^6\, :\,t^6 :\,t^7 :\,t^8)$ &  &  \\
{\rm (xix)} & or & 4 & 5\\
&  $(1+a_1t+a_2t^2+a_3t^4\, :\,  t^5+a_1t^6 :\, t^7\, : \, t^8 :\,t^{10} :\,t^{11})$ &  &  \\
& $(1+a_1t+a_2t^2+a_3t^4\, :\, t^3\,:\,t^5+a_1t^6\, :\,t^6 :\,t^7 :\,t^8)$ &  &  \\ \hline 
{\rm (xx)} & $(1+a_1t+a_2t^2+a_3t^3\, :\, t^5+a_4t^7\, :\, t^6\, : \, t^9 :\,t^{12} :\,t^{13})$   & 3 & 5 \\ 
&  $(1+a_1t+a_2t^2+a_3t^3\, :\, t+a_1t^2+(a_2-a_4)t^3\,:\,t^4\, :\,t^5+a_4t^7 :\,t^6\, :\,t^7)$  & & \\ \hline
 & $(1+a_1t+a_2t^2+a_3t^3+a_4t^4\, :\, t^5+a_5t^8\, :\, t^6+a_6t^8\, : \, t^7\ :\, t^{10}\, :\, t^{11})$   & & \\ 
{\rm (xxi)} &  $(1+a_1t+\cdots+a_4t^4\, :\, t+a_1t^2+b_2t^3+b_3t^4\,:\,t^5+a_5t^8\, :\,t^6+a_6t^8 :\,t^7 :\,t^8)$  & 3 & 5 \\ 
& where $b_2=(a_2-a_6)$ and $b_3=(a_3-a_5-a_1a_6)$ & & \\ \hline
{\rm (xxii)}& $(1+a_1t+a_2t^2+a_3t^3\, :\, t^4 :\, t^9\, :\,t^{10} :\,t^{11})$  & 3 & 4 \\ 
&  $(1+a_1t+a_2t^2+a_3t^3\, :\, t+a_1t^2+a_2t^3\, :\,t^2+a_1t^3 :\,t^4 :\,t^5 :\,t^6)$ & & \\ \hline
&  $(1+2a_4t+a_2t^2+a_3t^3\, :\, t^4+a_4t^5+a_5t^6\, : \, t^7 :\,t^{10} :\,t^{12} :\,t^{13})$ & & \\
& $(1+2a_4t+a_2t^2+a_3t^3 :\,t^3+a_4t^4+a_5^*t^5 :\,t^4+a_4t^5+a_5t^6\, :\,t^6 :\,t^7 :\,t^8)$ &  &  \\
{\rm (xxiii)} & or & 4 & 4\\
&  $(1+2a_4t+a_2t^2+a_3t^5\, :\, t^4+a_4t^5+a_5^{**}t^6\, : \, t^7 :\,t^{10} :\,t^{12} :\,t^{13})$ & & \\
& $(1+a_1t+(a_1^2-a_4^2)t^2+a_3t^5 :\,t^3+a_4t^4+a_4^2t^5 :\,t^4+a_4t^5+a_5^{***}t^6\, :\,t^6 :\,t^7 :\,t^8)$ &  &  \\ 
& where $a_5^*=(a_2-a_4^2-a_5),~a_5^{**}=(a_2+a_1a_4-a_1^2)$ and $a_5^{***}=(a_2-2a_4^2)$   & & \\ \hline

\end{tabular} 
$ $
\end{center}

\begin{proof} All follows from Theorems \ref{general} and \ref{general1}. Cases (i), (iii), (iv), (vii)-(ix) and  (xiii)-(xvi) follow from \ref{general}.(i); cases (ii), (x), (xvii) and (xxii) from \ref{general}.(ii); cases (vi), (xx) and (xxi) from \ref{general}.(iii); and cases (v), (xi), (xii), (xviii), (xix) and (xxiii) from \ref{general1}.
\end{proof}


\begin{center} \scshape References \end{center}
\begin{biblist}
\parskip = 0pt plus 2pt















\bib{KK}{article}{
   author={Altman, Allen B.},
   author={Kleiman, Steven L.},
   title={Compactifying the Jacobian},
   journal={Bull. Amer. Math. Soc.},
   volume={82},
   date={1976},
   number={6},
   pages={947--949},
   issn={0002-9904},
   review={\MR{429908}},
   doi={10.1090/S0002-9904-1976-14229-2},
}




\bib{BDF}{article}{
   author={Barucci, V.},
   author={D'Anna, M.},
   author={Fr\"{o}berg, R.},
   title={Analytically unramified one-dimensional semilocal rings and their
   value semigroups},
   journal={J. Pure Appl. Algebra},
   volume={147},
   date={2000},
   number={3},
   pages={215--254},
   issn={0022-4049},
   review={\MR{1747441}},
   doi={10.1016/S0022-4049(98)00160-1},
}



 \bib{BC}{article}{
   author={Behnke, Kurt},
   author={Christophersen, Jan Arthur},
   title={Hypersurface sections and obstructions (rational surface
   singularities)},
   journal={Compositio Math.},
   volume={77},
   date={1991},
   number={3},
   pages={233--268},
   issn={0010-437X},
   review={\MR{1092769}},
}








\bib{CFM1}{article}{
   author={Cotterill, Ethan},
   author={Feital, Lia},
   author={Martins, Renato Vidal},
   title={Dimension counts for cuspidal rational curves via semigroups},
   journal={Proc. Amer. Math. Soc.},
   volume={148},
   date={2020},
   number={8},
   pages={3217--3231},
   issn={0002-9939},
   review={\MR{4108832}},
   doi={10.1090/proc/15062},
}




\bib{CLM}{article}{
   author={Cotterill, Ethan},
   author={Lima, Vin\'{\i}cius Lara},
   author={Martins, Renato Vidal},
   title={Severi dimensions for unicuspidal curves},
   journal={J. Algebra},
   volume={597},
   date={2022},
   pages={299--331},
   issn={0021-8693},
   review={\MR{4406400}},
   doi={10.1016/j.jalgebra.2021.11.028},
}

\bib{EH}{article}{
   author={Eisenbud, David},
   author={Harris, Joe},
   title={On varieties of minimal degree (a centennial account)},
   conference={
      title={Algebraic geometry, Bowdoin, 1985},
      address={Brunswick, Maine},
      date={1985},
   },
   book={
      series={Proc. Sympos. Pure Math.},
      volume={46},
      publisher={Amer. Math. Soc., Providence, RI},
   },
   date={1987},
   pages={3--13},
   review={\MR{927946}},
   doi={10.1090/pspum/046.1/927946},
}


\bib{EHKS}{article}{
   author={Eisenbud, David},
   author={Koh, Jee},
   author={Stillman, Michael},
   title={Determinantal equations for curves of high degree},
   journal={Amer. J. Math.},
   volume={110},
   date={1988},
   number={3},
   pages={513--539},
   issn={0002-9327},
   review={\MR{944326}},
   doi={10.2307/2374621},
}










\bib{FM}{article}{
   author={Feital, Lia},
   author={Vidal Martins, Renato},
   title={Gonality of non-Gorenstein curves of genus five},
   journal={Bull. Braz. Math. Soc. (N.S.)},
   volume={45},
   date={2014},
   number={4},
   pages={649--670},
   issn={1678-7544},
   review={\MR{3296185}},
   doi={10.1007/s00574-014-0067-5},
}




\bibitem{GM} E. Gagliardi. R. V. Martins, {\it Max Noether Theorem for Singular Curves}, arXiv:2202.09349

\bibitem{GNM} N. Galdino, D. Nicolau, R. V. Martins, {\it Family of Rational Curves with Two Blocks of Values}, https://sites.google.com/view/singularintegralcurves/in$\%$C3$\%$ADcio










\bib{KM}{article}{
   author={Kleiman, Steven Lawrence},
   author={Martins, Renato Vidal},
   title={The canonical model of a singular curve},
   journal={Geom. Dedicata},
   volume={139},
   date={2009},
   pages={139--166},
   issn={0046-5755},
   review={\MR{2481842}},
   doi={10.1007/s10711-008-9331-4},
}






\bib{Mt}{article}{
   author={Martins, Renato Vidal},
   title={On trigonal non-Gorenstein curves with zero Maroni invariant},
   journal={J. Algebra},
   volume={275},
   date={2004},
   number={2},
   pages={453--470},
   issn={0021-8693},
   review={\MR{2052619}},
   doi={10.1016/j.jalgebra.2003.10.033},
}



\bib{LMS}{article}{
   author={Martins, Renato Vidal},
   author={Lara, Danielle},
   author={Souza, Jairo Menezes},
   title={On gonality, scrolls, and canonical models of non-Gorenstein
   curves},
   journal={Geom. Dedicata},
   volume={203},
   date={2019},
   pages={111--133},
   issn={0046-5755},
   review={\MR{4027587}},
   doi={10.1007/s10711-019-00428-2},
}













\bib{RS}{article}{
   author={Rosa, Renata},
   author={St\"{o}hr, Karl-Otto},
   title={Trigonal Gorenstein curves},
   journal={J. Pure Appl. Algebra},
   volume={174},
   date={2002},
   number={2},
   pages={187--205},
   issn={0022-4049},
   review={\MR{1921820}},
   doi={10.1016/S0022-4049(02)00122-6},
}


\bib{R}{article}{
   author={Rosenlicht, Maxwell},
   title={Equivalence relations on algebraic curves},
   journal={Ann. of Math. (2)},
   volume={56},
   date={1952},
   pages={169--191},
   issn={0003-486X},
   review={\MR{48856}},
   doi={10.2307/1969773},
}








\bib{Sc}{article}{
   author={Schreyer, Frank-Olaf},
   title={Syzygies of canonical curves and special linear series},
   journal={Math. Ann.},
   volume={275},
   date={1986},
   number={1},
   pages={105--137},
   issn={0025-5831},
   review={\MR{849058}},
   doi={10.1007/BF01458587},
}



 \bib{St}{article}{
   author={Stevens, J.},
   title={The versal deformation of universal curve singularities},
   journal={Abh. Math. Sem. Univ. Hamburg},
   volume={63},
   date={1993},
   pages={197--213},
   issn={0025-5858},
   review={\MR{1227874}},
   doi={10.1007/BF02941342},
}

\bib{S}{article}{
   author={St\"{o}hr, Karl-Otto},
   title={On the poles of regular differentials of singular curves},
   journal={Bol. Soc. Brasil. Mat. (N.S.)},
   volume={24},
   date={1993},
   number={1},
   pages={105--136},
   issn={0100-3569},
   review={\MR{1224302}},
   doi={10.1007/BF01231698},
}


















\end{biblist}
\end{document}